\documentclass[a4paper, 11pt, leqno]{article}
\usepackage{amsmath,amsthm,mathrsfs}
\usepackage{amsfonts,amssymb}

\newtheorem{theorem}{Theorem}
\newtheorem{proposition}{Proposition}
\newtheorem{definition}{Definition}
\newtheorem{lemma}{Lemma}

\newtheorem{remark}{Remark}

\begin{document}

\title{\bf Hardy spaces associated with Schr\"odinger operators on the Heisenberg group}
\author{Chin-Cheng Lin\footnote{Corresponding author.}\,\ \footnote{Supported by 
       National Science Council of Taiwan under Grant \#NSC 97-2115-M-008-021-MY3.}\ ,\ \
       Heping Liu\footnote{Supported by National Natural Science Foundation of China 
       under Grant \#10871003, \#10990012 and the Specialized Research Fund for 
       the Doctoral Program of Higher Education of China under Grant \#2007001040.}\ ,\ \ and\ \ Yu Liu}
\date{}
\maketitle

\begin{abstract} 
Let $L= -\Delta_{\mathbb{H}^n}+V$ be a
Schr\"odinger operator on the Heisenberg group $\mathbb{H}^n$,
where $\Delta_{\mathbb{H}^n}$ is the sub-Laplacian and the
nonnegative potential $V$ belongs to the reverse H\"older class
$B_{\frac{Q}{2}}$ and $Q$ is the homogeneous dimension of
$\mathbb{H}^n$. The Riesz transforms associated with the Schr\"odinger
operator $L$ are bounded from $L^1(\mathbb{H}^n)$ to
$L^{1,\infty}(\mathbb{H}^n)$. The $L^1$ integrability of the Riesz
transforms associated with $L$ characterizes a certain Hardy type space
denoted by $H^1_L(\mathbb{H}^n)$ which is larger than the usual
Hardy space $H^1(\mathbb{H}^n)$. We define $H^1_L(\mathbb{H}^n)$
in terms of the maximal function with respect to the semigroup
$\big \{ e^{-s L}:\; s>0 \big\}$, and give the atomic
decomposition of $H^1_L(\mathbb{H}^n)$. As an application of the
atomic decomposition theorem, we prove that $H^1_L(\mathbb{H}^n)$
can be characterized by the Riesz transforms associated with $L$. All results
hold for stratified groups as well.


\end{abstract}
\section {Introduction}

The Schr\"odinger operators with a potential satisfying the
reverse H\"older inequality have been studied by various authors.
Some basic results on the Euclidean spaces were established by
Fefferman \cite{Fefferman}, Shen\cite{Shen}, and Zhong \cite{Zhong}. 
The extension to a more general setting was given by
Lu \cite{Lu} and Li \cite{Li}. In this article we consider the
Schr\"odinger operator $L= -\Delta_{\mathbb{H}^n}+V$ on the
Heisenberg group $\mathbb{H}^n$, where $\Delta_{\mathbb{H}^n}$ is
the sub-Laplacian and the nonnegative potential $V$ belongs to the
reverse H\"older class $B_q, q \geq \frac{Q}{2}$. Here $Q$ is the
homogeneous dimension of $\mathbb{H}^n$. 
Let $R^L_j = X_j L^{-\frac{1}{2}},\, j=1, \cdots, 2n$, 
be the Riesz transforms associated with the Schr\"odinger operator $L$, 
where $X_j$'s are left-invariant
vector fields generating the Lie algebra of $\mathbb{H}^n$. If $q
\geq Q$, $R^L_j$ are Calder\'on--Zygmund operators (cf.
\cite{Lu}). When $\frac{Q}{2} \leq q < Q$, $R^L_j$ are bounded on
$L^p(\mathbb{H}^n)$ for $1<p\leq \frac{Qq}{Q-q}$ (cf. \cite{Li}).
In the current paper, we prove, by providing a counterexample,
that the above estimate of the range of $p$ is sharp.

It is well known that $R^L_j$ might not be Calder\'on--Zygmund
operators. However, these operators $R^L_j$ do observe some
boundedness conditions: they are bounded from $L^1(\mathbb{H}^n)$
to $L^{1,\infty}(\mathbb{H}^n)$ (see Theorem \ref{thm2} below),
and bounded from $H^1(\mathbb{H}^n)$ to $L^1(\mathbb{H}^n)$ (see
Remark \ref{rem6} below). We should remark that, unlike the
classical case, the conditions $f \in L^1(\mathbb{H}^n)$ and
$R^L_jf \in L^1(\mathbb{H}^n)$ ($j=1, \cdots, 2n$) do not ensure
$f \in H^1(\mathbb{H}^n)$. In view of these new complication, we
will introduce the notion of the Hardy space $H^1_L(\mathbb{H}^n)$
associated with $L$ in terms of the maximal function with respect
to the semigroup $\big\{ e^{-s L}: s>0 \big\}$. The atomic
decomposition of $H^1_L(\mathbb{H}^n)$ will be given and, as an
application, we prove that $H^1_L(\mathbb{H}^n)$ can be
characterized by the Riesz transforms $R^L_j$ associated with the
Schr\"odinger operator $L$. We also remark that the Heisenberg
group is a typical case of stratified groups. Our all results
can be established for stratified groups by the same arguments.

The current work is inspired by the pioneering work of
Dziuba\'nski and Zienkiewicz \cite{Dziubanski},
in which the Hardy space associated with the
Schr\"odinger operator on the Euclidean spaces was studied.
Equipped with some enhanced technique,
we will establish some extended results in the setting of
the Heisenberg group. For example, we give some descriptions
of kernels (Section 3) and develop the theory of local Hardy
spaces (Section 4). The corresponding results to these
on Euclidean spaces have been already known. On
Euclidean spaces the local Riesz transforms characterization of
local Hardy spaces is obtained via subharmonicity (cf.
\cite{Goldberg}). However, this kind of approach fails on the
Heisenberg group as pointed out in \cite{Christ}, so we have to
develop a new method. Our approach is to decompose a
function $f=\widetilde{f}+(f- \widetilde{f})$ such that
$\widetilde{f}$ is in the local Hardy space $h^1(\mathbb{H}^n)$ and
the Riesz transforms of $(f- \widetilde{f})$ are controlled by the
local Riesz transforms of $f$. We then can use the Riesz transforms
characterization of Hardy space $H^1$ on stratified groups given by
Christ and Geller (cf. \cite{Christ}).
We establish the atomic decomposition
theorem for general $H^{1,q}_L$-atoms rather than
$H^{1, \infty}_L$-atoms as in \cite{Dziubanski},
because it is more convenient
to study the dual space of $H^1_L(\mathbb{H}^n)$, which will be
dealt with in the forthcoming paper \cite{Lin}. We also prove the weak
$(1, 1)$ boundedness of Riesz transforms $R^L_j$, which is useful
to establish the $H^1_L-L^1$ boundedness of $R^L_j$.

This article is organized as follows. In Section 2, we set 
notations and state our main results. In Section 3 we give 
estimates of kernels of the semigroup $\big\{ e^{-s L}: s>0 \big\}$ 
and the Riesz transforms, which will be used in the sequel. 
Most proofs in this section are inspired from \cite{Dziubanski} 
and \cite{Shen}; however, they are new in our setting.
In Section 4 we discuss local Hardy spaces
$h^1(\mathbb{H}^n)$. Note that an $H^1_L$-function is locally
equal to a function in a certain scaled local Hardy space.
Specifically, $H^1_L(\mathbb{H}^n)$ coincides with the local Hardy
space $h^1(\mathbb{H}^n)$ if there exists a positive number $C$
such that $\frac{1}{C} \leq V \leq C$ (see Remark \ref{rem5}
below). In Section 5 we establish the atomic decomposition of
$H^1_L(\mathbb{H}^n)$. Section 6 is devoted to the Riesz
transforms $R^L_j$. We prove that $R^L_j$ are bounded from
$L^1(\mathbb{H}^n)$ to $L^{1,\infty}(\mathbb{H}^n)$, and they
characterize $H^1_L(\mathbb{H}^n)$. The counterexample mentioned
above is also given in Section 6. Finally, we include in Section 7
a brief discussion the corresponding results for stratified 
groups without proofs.

Throughout this article, we will use $C$ to denote a positive
constant, which is independent of main parameters and not
necessarily the same at each occurrence. By $A \sim B$, we mean that
there exists a constant $C>1$ such that $\frac{1}{C} \leq
\frac{A}{B}\leq C$. Moreover, we denote the conjugate exponent
of $q>1$ by $q'=q/(q-1)$.


\section {Notations and main results}

We recall some basic facts on the Heisenberg group, which are
easily found in many references. The $(2n+1)$-dimensional Heisenberg group
$\mathbb H^n$ is the Lie group with underlying manifold
$\mathbb R^{2n} \times \mathbb R$ and multiplication
\begin{eqnarray*}
(x, t)(y, s)= \big ( x+y,\, t+s+2 \sum_{j=1}^{n} (x_{n+j} y_{j} -
x_{j} y_{n+j}) \big ).
\end{eqnarray*}
A basis for the Lie algebra of left-invariant vector fields on
$\mathbb H^n$ is given by
\begin{eqnarray*}
X_{2n+1}= \frac{\partial}{\partial t},\ \
X_j= \frac{\partial}{\partial x_j}+2x_{n+j}
\frac{\partial}{\partial t},\ \ X_{n+j}=\frac{\partial}{\partial
x_{n+j}}-2x_j \frac{\partial}{\partial t},\ \ j=1, \cdots, n.
\end{eqnarray*}
All non-trivial commutators are $[X_j, X_{n+j}]= -4 X_{2n+1},\, j=1, \cdots, n$.
The sub-Laplacian $\Delta_{\mathbb H^n}$ and the gradient
$\nabla_{\mathbb H^n}$ are defined respectively by
\begin{eqnarray*} \Delta_{\mathbb H^n}= \sum^{2n}_{j=1} X^2_j\qquad
\text{and}\qquad \nabla_{\mathbb H^n}= (X_1, \cdots, X_{2n}).
\end{eqnarray*}
The dilations on $\mathbb H^n$ have the form
\begin{eqnarray*}
\delta_r (x, t)= ( rx, r^2 t), \qquad r>0.
\end{eqnarray*}
The Haar measure on $\mathbb H^n$ coincides with the Lebesgue
measure on $\mathbb R^{2n} \times \mathbb R$. The measure of any
measurable set $E$ is denoted by $|E |$. We define a homogeneous
norm on $\mathbb H^n$ by
\begin{eqnarray*}
|g|\, = \big (|x|^4+ |t|^2 \big )^\frac{1}{4}, \qquad g= (x,t) \in
\mathbb H^n.
\end{eqnarray*}
This norm satisfies the triangle inequality and leads to a
left-invariant distance $d(g,h) = |g^{-1}h|$. Then the
ball of radius $r$ centered at $g$ is given by
\begin{eqnarray*}
B(g,r)=\{h \in\mathbb H^n: \; |g^{-1}h |\, <r\}.
\end{eqnarray*}
There is a positive constant $b_1$ such that
\begin{eqnarray*}
\big| B(g,r) \big| = b_1 r^{Q}
\end{eqnarray*}
where $Q=2n+2$ is the homogeneous dimension of $\mathbb H^n$.
Exactly,
\begin{eqnarray*}
b_1= \big| B(0,1) \big| = \frac{2 \pi^{n+\frac{1}{2}} \Gamma
(\frac{n}{2})} {(n+1) \Gamma (n) \Gamma (\frac{n+1}{2})},
\end{eqnarray*}
but it is not important for us.

Now we turn to the Schr\"odinger operator 
$$L= -\Delta_{\mathbb{H}^n}+V.$$ 
A nonnegative
locally $L^q$ integrable function $\,V\,$ on $\mathbb H^n$ is said
to belong to $B_q\ (1<q<\infty)$ if there exists $C>0$ such that
the reverse H\"older inequality
\begin{eqnarray*}
\left ( \frac{1}{|B|} \int_B V(g)^q \, dg \right )^{\frac{1}{q}}
\leq C \left ( \frac{1}{|B|} \int_B V(g) \, dg \right )
\end{eqnarray*}
holds for every ball $B$ in $\mathbb H^n$. Obviously, $B_{q_{_1}}
\subset B_{q_{_2}}$ if $q_{_1} > q_{_2}$. But it is important that
the $B_q$ class has a property of ``self improvement"; that is, if
$V\in B_q$, then $V\in B_{q+\varepsilon}$ for some $\varepsilon
>0$. In this article we always assume that $0 \not \equiv V\in
B_{\frac{Q}{2}}$ and then $V \in B_{q_{_0}}$ for some $q_{_0}>
\frac{Q}{2}$. Of course we may assume that $q_{_0}<Q$.

Let $\{ T_s:\; s>0 \} = \big\{ e^{s\Delta_{\mathbb H^n}}:\; s>0
\big\}$ be the heat semigroup with the convolution kernel
$H_s(g)$. The heat kernel $H_s(g)$ satisfies the estimate
\begin{eqnarray}\label{a1}
0< H_s(g) \leq C\, s^{-\frac{Q}{2}} e^{-A_0 s^{-1} |g|^2},
\end{eqnarray}
where $A_0$ is a positive constant (cf. \cite{Jerison}). An
explicit expression of $H_s(g)$ in terms of the Fourier transform
with respect to central variable was given by Hulanicki
\cite{Hulanicki}. Because $V \geq 0$ and $V \in
L^{\frac{Q}{2}}_{\mathrm{loc}}(\mathbb H^n)$, the Schr\"odinger
operator $L$ generates a $(C_0)$ contraction semigroup $\big\{
T^L_s:\; s>0 \big\} = \big\{ e^{-s L}:\; s>0 \big\}$. Let
$K^L_s(g, h)$ denote the kernel of $T^L_s$. By the Trotter product
formula (cf. \cite{Goldstein}),
\begin{eqnarray}\label{a2}
0 \leq K^L_s(g, h) \leq H_s(g,h)=H_s(h^{-1}g).
\end{eqnarray}

Let us consider the maximal functions with respect to the
semigroups $\{ T_s:\; s>0 \}$ and $\big\{ T^L_s:\; s>0 \big\}$
defined by
\begin{eqnarray*}
Mf(g) &=& \sup _{s>0} \big| T_sf(g) \big|,\\
M^Lf(g) &=& \sup _{s>0} \big| T^L_sf(g) \big|.
\end{eqnarray*}
It is well known that the maximal function $Mf$ characterizes the
Hardy space $H^1(\mathbb{H}^n)$; that is, $f \in H^1(\mathbb{H}^n)$ 
if and only if $Mf \in L^1(\mathbb{H}^n)$, and
$\big\| f \big\|_{H^1} \sim \big\| Mf \big\|_{L^1}$ (cf.
\cite{Folland-Stein}).

We define the Hardy space $H^1_L(\mathbb{H}^n)$ associated
with the Schr\"odinger operator $L$ as follows.

\begin{definition}\label{def1}
{\rm A function $f\in L^1(\mathbb{H}^n)$ is said to be in $H^1_L(\mathbb{H}^n)$
if the maximal function $M^Lf$ belongs to $L^1(\mathbb{H}^n)$.
The norm of such a function is defined by
$\big\| f \big\|_{H^1_L} = \big\|M^Lf \big\|_{L^1}$.}
\end{definition}

It is visible from $(\ref{a2})$ that the space
$H^1_L(\mathbb{H}^n)$ is larger than the usual Hardy space
$H^1(\mathbb{H}^n)$. This fact will be seen from the atomic
decompositions of $H^1_L(\mathbb{H}^n)$.

We define the auxiliary function $\rho(g, V)= \rho(g)$ by
\begin{eqnarray*}
\rho(g)= \sup_{r>0}\, \bigg \{ r:\; \frac{1}{r^{Q-2}}\int_{B(g,
r)}V(h)\, dh \leq 1 \bigg \}, \qquad g\in\mathbb H^n.
\end{eqnarray*}
This kind of auxiliary function was introduced by Shen in \cite{Shen1994}
for the potential $V$ satisfying $\max_{x\in B} V(x)\le \frac C{|B|}\int_B V(x) dx$.
Properties of the auxiliary function $\rho(g, V)$ are given
by Shen \cite{Shen} on Euclidean spaces and by Lu \cite{Lu} on
homogeneous spaces (their auxiliary function
$m(g, V)= \frac{1}{\rho(g, V)}$ in practice). It is
known that $0< \rho(g)< \infty$ for any $g \in \mathbb{H}^n$
(from Lemma \ref{lem2} in Section 3).

\begin{definition}\label{def2}\
{\rm Let $1<q \leq \infty$. A function $a\in L^q(\mathbb{H}^n)$ is called an
$H^{1,q}_L$-atom if the following conditions hold:
\begin{itemize}
\item [\rm(i)] $\mathrm{supp}\, a \subset B(g_0,r),$
\item [\rm(ii)] $\big\| a \big\|_{L^{q}} \leq \big| B(g_0,r) \big|^{\frac 1q -1},$
\item [\rm(iii)] if $\ r < \rho(g_0),$ then $\displaystyle\int_{B(g_0,r)} a(g)\, dg =0.$
\end{itemize}}
\end{definition}

\begin{theorem}\label{thm1}
Let $f \in L^1(\mathbb{H}^n)$ and $1<q \leq \infty$.
Then $f \in H^1_L(\mathbb{H}^n)$ if and only if
$\,f$ can be written as $f= \sum_j \lambda_j\, a_j$,
where $a_j$ are $H^{1,q}_L$-atoms, $\sum_j |\lambda_j|< \infty$,
and the sum converges in $H^1_L(\mathbb{H}^n)$ norm. Moreover,
\begin{eqnarray*}
\big\| f \big\|_{H^1_L} \sim \inf \bigg\{ \sum_j |\lambda_j| \bigg\},
\end{eqnarray*}
where the infimum is taken over all atomic decompositions of $\,f$
into $H^{1,q}_L$-atoms.
\end{theorem}

The Riesz transforms $R^L_j$ associated with the Schr\"odinger operator $L$
are defined by
\begin{eqnarray*}
R^L_j= X_j L^{-\frac{1}{2}},\qquad j=1, \cdots, 2n.
\end{eqnarray*}
As pointed out at the beginning, $R^L_j$ are bounded on
$L^p(\mathbb{H}^n)$ for $1<p\leq p_{_0}$ where
$\frac{1}{p_{_0}}=\frac{1}{q_{_0}}-\frac{1}{Q}$. 
A counterexample will be given in Section 6 to show that the above range
of $p$ is optimal. For $p=1$, we have the following weak type estimate.

\begin{theorem}\label{thm2}
The Riesz transforms $R^L_j$ are bounded from $L^1(\mathbb{H}^n)$
to $L^{1,\infty}(\mathbb{H}^n)$.
\end{theorem}

The Riesz transforms $R^L_j$ are also bounded from
$H^1(\mathbb{H}^n)$ to $L^1(\mathbb{H}^n)$ (see Remark \ref{rem6} below).
However, these operators do not characterize the usual Hardy space
$H^1(\mathbb{H}^n)$. They characterize $H^1_L(\mathbb{H}^n)$ which
is larger than $H^1(\mathbb{H}^n)$.

\begin{theorem}\label{thm3}
A function $f \in H^1_L(\mathbb{H}^n)$ if and only if $\,f \in
L^1(\mathbb{H}^n)$ and $R^L_jf \in L^1(\mathbb{H}^n),\, j=1,
\cdots, 2n$. Moreover,
\begin{eqnarray*}
\big\| f \big\|_{H^1_L} \sim \big\| f \big\|_{L^1} +
\sum_{j=1}^{2n} \big\| R^L_jf \big\|_{L^1}.
\end{eqnarray*}
\end{theorem}

The proof of Theorem \ref{thm1} will be given at the end of Section 5, while the proofs
of Theorems \ref{thm2} and \ref{thm3} will be given in Section 6.


\section {Estimates of the kernels}

In this section we give some estimates of kernels of the
semigroup $\{ T^L_s \}$ and the Riesz transforms $R^L_j$, which
will be used in the sequel. 
The proofs of Lemmas $\ref{lem6}-\ref{lem9}$ will closely follow the 
arguments on $\mathbb{R}^n$ as presented by Shen in \cite{Shen}.
First we collect some basic facts about
the potential $V$ satisfying the reverse H\"older inequality. We
assume that $V \in B_{q_{_0}}$ for some $q_{_0}>
\frac{Q}{2}$. We may add a restriction that $q_{_0}<Q$ when
necessary. In this case, we assume the relation
$\frac{1}{p_{_0}}=\frac{1}{q_{_0}}-\frac{1}{Q}$.

\begin{lemma}\label{lem1} The measure $V(h)\,dh$ satisfies the doubling condition;
that is, there exists $C >0$ such that
\begin{equation*}
\int_{B(g,2r)} V(h)\, dh \leq C \int_{B(g,r)} V(h)\, dh
\end{equation*}
for all balls $B(g,r)$ in $\mathbb{H}^n$.
\end{lemma}

\begin{lemma}\label{lem2}
For $0<r<R<\infty,$
\begin{eqnarray*}
\frac{1}{r^{Q-2}}\int_{B(g,r)}V(h)\, dh\leq C \bigg ( \frac{r}{R}
\bigg )^{2-\frac{Q}{q_{_0}}} \frac{1}{R^{Q-2}}\int_{B(g, R)}V(h)\,
dh.
\end{eqnarray*}
\end{lemma}

\begin{lemma}\label{lem3} If $\, r= \rho(g)$, then
\begin{eqnarray*}
\frac{1}{r^{Q-2}}\int_{B(g,r)}V(h)\, dh=1.
\end{eqnarray*}
Moreover,
\begin{eqnarray*}
\frac{1}{r^{Q-2}}\int_{B(g, r)}V(h)\, dh \sim 1 \qquad
\text{if and only if} \qquad  r \sim \rho (g).
\end{eqnarray*}
\end{lemma}

\begin{lemma}\label{lem4}
There exists $l_0>0$ such that, for any $g$ and $h$ in $\mathbb H^n$,
\begin{eqnarray*}
\frac{1}{C} \bigg ( 1+ \frac{|h^{-1} g|}{\rho(g)} \bigg )^{-l_0}
\leq \frac{\rho(h)}{\rho(g)} \leq C \bigg (1+ \frac{|h^{-1}
g|}{\rho(g)} \bigg )^{\frac{l_0}{l_0+1}}.
\end{eqnarray*}
In particular, $\rho(h) \sim \rho(g)$ if $\,|h^{-1} g|< C\,
\rho(g)$.
\end{lemma}

\begin{lemma}\label{lem5} There exists $l_1>0$ such that, for any $g \in \mathbb H^n$,
\begin{eqnarray*}
\int_{B(g, R)}\frac{V(h)}{|h^{-1}g|^{Q-2} }\, dh \leq
\frac{C}{R^{Q-2}}\int_{B(g, R)}V(h)\, dh \leq C \bigg ( 1+
\frac{R}{\rho(g)} \bigg )^{l_1}.
\end{eqnarray*}
\end{lemma}

For the proofs of Lemmas \ref{lem1}$-$\ref{lem5}, we refer readers to
\cite{Lu}.

Now we turn to the estimates of kernels. Let $\Gamma(g,h, \tau)$
denote the fundamental solution for the operator $-\Delta_{\mathbb
H^n}+i\tau$, where $\tau \in \mathbb R$. For any $l>0$, there
exists $C_l>0$ such that
\begin{eqnarray}
\big| \Gamma (g, h, \tau) \big| &\leq& \frac{C_l } {\big (
1+|h^{-1}g|\, |\tau|^{\frac{1}{2}} \big )^l} \frac{1}{|h^{-1}g|^{Q-2} },\label{a3}\\
\big| \nabla_{{\mathbb H^n},g} \Gamma (g, h, \tau) \big| &\leq&
\frac{C_l } { \big (1+|h^{-1}g|\, |\tau|^{\frac{1}{2}} \big )^l
}\frac{1}{|h^{-1}g|^{Q-1} },\label{a4}
\end{eqnarray}
where $\nabla_{{\mathbb H^n},g}$ denote the gradient for
variable $g$. The estimate $(\ref{a4})$ still holds for $\nabla_{{\mathbb
H^n},h}$ instead of $\nabla_{{\mathbb H^n},g}$. The estimates
$(\ref{a3})$ and $(\ref{a4})$ are easily reduced from the
corresponding estimates of heat kernel (cf. \cite{Jerison}). We
remark that the explicit expression of $\Gamma(g, h)= \Gamma(g, h,
0)$ is obtained by Folland \cite{Folland}:
\begin{eqnarray*}
\Gamma (g, h)= \frac{2^{n-2} \Gamma (\frac{n}{2})^2}{\pi^{n+1}}
\frac{1}{|h^{-1}g|^{Q-2}}.
\end{eqnarray*}
Let $\Gamma^L (g, h, \tau)$ denote the fundamental solution for
the operator $L+i\tau$, where $\tau \in \mathbb R$. For any $l>0$,
there exists $C_l>0$ such that
\begin{eqnarray*}
\big| \Gamma^L (g, h, \tau) \big| \leq \frac{C_l}{\big
(1+|h^{-1}g|\, |\tau|^\frac{1}{2} \big )^l \big ( 1+|h^{-1}g|
(\rho(g)^{-1} \big )^l} \frac{1}{|h^{-1}g|^{Q-2}}
\end{eqnarray*}
(cf. \cite[Theorem 4.8]{Lu}). Since $\Gamma^L (g, h, \tau)=
\Gamma^L (h, g, -\tau)$, we have
\begin{eqnarray}\label{a5}
\big| \Gamma^L (g, h, \tau) \big| \leq \frac{C_l}{\big
(1+|h^{-1}g|\, |\tau|^\frac{1}{2} \big )^l \big ( 1+|h^{-1}g|
(\rho(g)^{-1}+ \rho(h)^{-1}) \big )^l} \frac{1}{|h^{-1}g|^{Q-2}}.
\end{eqnarray}

\begin{lemma}\label{lem6}
For any $l>0$, there exists $C_l>0$ such that if $\,|h^{-1}g|\leq
\rho(g)$, then
\begin{eqnarray*}
\big| \Gamma^L(g, h, \tau)- \Gamma (g, h, \tau) \big| \leq
\frac{C_l}{\big ( 1+|h^{-1}g|\, |\tau|^{\frac{1}{2}}\big )^l}
\frac{1}{\rho(g)^{\delta} |h^{-1}g|^{Q-2- \delta}},
\end{eqnarray*}
where $\delta = 2-\frac{Q}{q_{_0}} >0$.
\end{lemma}

\begin{proof} Note that
\begin{eqnarray*}-\Delta_{\mathbb H^n, g} \big( \Gamma^L(g, h, \tau)-\Gamma (g, h,
\tau) \big) +i\tau \big( \Gamma^L(g, h, \tau)-\Gamma (g, h, \tau)
\big)= -V(g)\Gamma^L(g, h, \tau),
\end{eqnarray*}
where $-\Delta_{\mathbb H^n, g}$ denotes the sub-Laplacian for
variable $g$. Since $\Gamma (g, h, \tau)$ is the fundamental
solution for $-\Delta_{\mathbb H^n }+i\tau$, we have
\begin{eqnarray}\label{a6}
\Gamma^L(g, h, \tau) -\Gamma (g, h, \tau)= -\int_{\mathbb H^n}
\Gamma (g, w, \tau)V(w)\Gamma^L(w, h, \tau)\, dw.
\end{eqnarray}
Let $R=|h^{-1}g| \leq \rho (g)$. By $(\ref{a3})$ and $(\ref{a5})$,
\begin{eqnarray*}
\big| \Gamma^L(g, h, \tau) -\Gamma(g, h, \tau) \big| &\leq&
\int_{\mathbb H^n} \frac{C_l}{\big ( 1+|g^{-1}
w|\, |\tau|^{\frac{1}{2}}\big )^l |g^{-1}w|^{Q-2}} \\
& & \qquad \cdot\frac{V(w)\, dw}{\big ( 1+|h^{-1}w|\,
|\tau|^{\frac{1}{2}}\big )^l \big ( 1+|h^{-1}w|\, \rho(h)^{-1}
\big)^l |h^{-1} w|^{Q-2}}\\
&=& \int_{|g^{-1}w|< \frac{R}{2}} + \int_{|h^{-1}w|< \frac{R}{2}}
+ \int_{|g^{-1}w|\geq \frac{R}{2},\, |h^{-1}w|\geq \frac{R}{2} }\\
&=& I_1+I_2+I_3.
\end{eqnarray*}
By Lemma \ref{lem5} and Lemma \ref{lem2}, we have
\begin{eqnarray*}
I_1 &\leq& \frac{C_l}{\big ( 1+ R\, |\tau|^{\frac{1}{2}}\big )^l
R^{Q-2}} \int_{B(g, \frac{R}{2})} \frac{V(w)\, dw}{|g^{-1}w|^{Q-2}}\\
&\leq& \frac{C_l}{\big ( 1+ R\, |\tau|^{\frac{1}{2}}\big )^l
R^{Q-2}} \frac{1}{R^{Q-2}} \int_{B(g, \frac{R}{2})} V(w)\, dw\\
&\leq& \frac{C_l}{\big ( 1+ R\, |\tau|^{\frac{1}{2}}\big )^l
R^{Q-2}} \bigg ( \frac{R}{\rho(g)} \bigg )^{2- \frac{Q}{q_{_0}}}.
\end{eqnarray*}
Similarly,
\begin{eqnarray*}
I_2 \leq \frac{C_l}{\big ( 1+ R\, |\tau|^{\frac{1}{2}}\big )^l
R^{Q-2}} \bigg ( \frac{R}{\rho(g)} \bigg )^{2- \frac{Q}{q_{_0}}}.
\end{eqnarray*}
Note that $|g^{-1}w| \sim |h^{-1}w|$ when $|g^{-1}w|\geq
\frac{R}{2}$ and $|h^{-1}w|\geq \frac{R}{2}$. It yields
\begin{eqnarray*}
I_3 &\leq& \frac{C_l}{\big ( 1+ R\, |\tau|^{\frac{1}{2}}\big )^l}
\int_{|h^{-1}w|\geq \frac{R}{2}} \frac{V(w)\, dw}{\big (
1+|h^{-1}w|\, \rho(h)^{-1} \big )^l |h^{-1}w|^{2(Q-2)}}\\
&\leq& \frac{C_l}{\big ( 1+ R\, |\tau|^{\frac{1}{2}}\big )^l}
\bigg ( \int_{\rho(h)>|h^{-1}w|\geq \frac{R}{2}} \frac{V(w)\,
dw}{|h^{-1}w|^{2(Q-2)}} + \rho(h)^l  \int_{|h^{-1}w|\geq \rho(h)}
\frac{V(w)\, dw}{|h^{-1}w|^{2(Q-2)+l}} \bigg ).
\end{eqnarray*}
We may assume $\rho(h)> \frac{R}{2}$. Otherwise,
$\rho(g) \sim \rho(h) \sim |h^{-1}g|$ and Lemma \ref{lem6} is
obviously true. By H\"older's inequality, $B_{q_{_0}}$ condition
and Lemma \ref{lem3}, we obtain
\begin{eqnarray*}
&&\int_{\rho(h)>|h^{-1}w|\geq \frac{R}{2}} \frac{V(w)\,dw}
  {|h^{-1}w|^{2(Q-2)}}\\
&&\qquad\leq C \bigg ( \int_{B(h, \rho(h))}
  V(w)^{q_{_0}}\, dw\bigg )^\frac{1}{q_{_0}} \bigg (
  \int_{\frac{R}{2}}^{\rho(h)} t^{-2(Q-2)q'_{_0}
  +Q-1}\, dt \bigg )^\frac{1}{q'_{_0}}\\
&&\qquad\leq C\, \rho(h)^{\frac{Q}{q_{_0}}-2} R^{-2(Q-2)+
  \frac{Q}{q'_{_0}}}\\
&&\qquad = \frac{C}{R^{Q-2}} \bigg ( \frac{R}{\rho(h)}
  \bigg )^{2- \frac{Q}{q_{_0}}}.
\end{eqnarray*}
Using Lemma \ref{lem5} and taking $l$ sufficiently large, we
obtain
\begin{eqnarray*}
\rho(h)^l  \int_{|h^{-1}w|\geq \rho(h)} \frac{V(w)\,
dw}{|h^{-1}w|^{2(Q-2)+l}} &\leq& C_l\, \rho(h)^l
\sum_{j=1}^{\infty}
(2^j \rho(h))^{-2(Q-2)-l} \int_{B(h, 2^j \rho(h))} V(w)\, dw\\
&\leq& \frac{C_l}{\rho(h)^{Q-2}} \sum_{j=1}^{\infty}
2^{-(l-l_1+Q-2)j} \\
&\leq& \frac{C_l}{\rho(h)^{Q-2}} \\
&\leq& \frac{C_l}{R^{Q-2}} \bigg (
\frac{R}{\rho(h)} \bigg )^{2- \frac{Q}{q_{_0}}}.
\end{eqnarray*}
By Lemma \ref{lem4}, $\rho(h) \sim \rho(g)$ when $|h^{-1}g|\leq
\rho(g)$. Lemma \ref{lem6} is proved.
\end{proof}

\begin{lemma}\label{lem7}
For any $l>0$, there exists $C_l>0$ such that
\begin{eqnarray}\label{a7}
K^L_s(g,h) \leq \frac{C_l}{\big ( 1+|h^{-1}g| (\rho(g)^{-1}+
\rho(h)^{-1}) \big )^l} \frac{1}{|h^{-1}g|^{Q}}.
\end{eqnarray}
Let $\,A \geq 1$ be a fixed constant. If $\,|h^{-1}g|\leq A\,
\rho(g)$, then
\begin{eqnarray}\label{a8}
\big| K^L_s(g,h)-H_s(g,h) \big| \leq \frac{C}{\rho(g)^{\delta}
|h^{-1}g|^{Q- \delta}},
\end{eqnarray}
where $\delta = 2-\frac{Q}{q_{_0}} >0$.
\end{lemma}

\begin{proof}
Note that
\begin{eqnarray*}
K^L_s(g,h)&=& \frac{1}{2 \pi} \int_{-\infty}^{\infty} e^{is \tau}
\Gamma^L(g, h, \tau)\, d\tau,\\
H_s(g,h)&=& \frac{1}{2 \pi} \int_{-\infty}^{\infty} e^{is \tau}
\Gamma (g, h, \tau)\, d\tau.
\end{eqnarray*}
Then $(\ref{a7})$ follows from $(\ref{a5})$ by integration. By the same
way, $(\ref{a8})$ follows from Lemma \ref{lem6} when $A=1$. Since
\begin{eqnarray*}
\big| K^L_s(g,h)-H_s(g,h) \big| \leq 2 H_s(g,h) \leq
\frac{C}{|h^{-1}g|^{Q}},
\end{eqnarray*}
$(\ref{a8})$ still holds for $\,A >1$.
\end{proof}

Let
\begin{eqnarray*}
R_j= X_j (-\Delta_{\mathbb{H}^n})^{-\frac{1}{2}},\qquad j=1,
\cdots, 2n,
\end{eqnarray*}
be the usual Riesz transforms with the convolution kernels
$R_j(g)$. We denote the kernel of $R^L_j$ by $R^L_j(g,h)$ and
write $R_j(g,h)=R_j(h^{-1}g)$. Let
\begin{eqnarray*}
\mathcal{R}^L=(R^L_1, \cdots, R^L_{2n})= \nabla_{\mathbb
H^n}L^{-\frac{1}{2}}, \qquad \mathcal{R}=(R_1, \cdots, R_{2n})=
\nabla_{\mathbb H^n}(-\Delta_{\mathbb{H}^n})^{-\frac{1}{2}}.
\end{eqnarray*}
The kernels of $\mathcal{R}^L$ and $\mathcal{R}$ are denoted by
$\mathcal{R}^L(g,h)$ and $\mathcal{R}(g,h)$, respectively. By
functional calculus, we have
\begin{eqnarray}\label{a9}
L^{-\frac{1}{2}}= \frac{1}{2\pi} \int_{\mathbb
R}(-i\tau)^{-\frac{1}{2}}(L+i\tau)^{-1}\, d\tau.
\end{eqnarray}
Thus,
\begin{eqnarray}\label{a10}
\mathcal{R}^L(g,h)= \frac{1}{2\pi} \int_{\mathbb
R}(-i\tau)^{-\frac{1}{2}}\nabla_{\mathbb H^n,g}\Gamma^L (g, h,
\tau)\, d\tau.
\end{eqnarray}
Similarly,
\begin{eqnarray}\label{a11}
\mathcal{R}(g,h)= \frac{1}{2\pi} \int_{\mathbb
R}(-i\tau)^{-\frac{1}{2}}\nabla_{\mathbb H^n,g}\Gamma (g, h,
\tau)\, d\tau.
\end{eqnarray}

\begin{lemma}\label{lem8}
Let $A>0$ be a fixed constant. Then
\begin{eqnarray*}
\int_{|h^{-1}g| \geq A\, \rho(h)} \big| R^L_j(g, h) \big|\, dg
\leq C.
\end{eqnarray*}
\end{lemma}

\begin{proof}
Let us fix $g_0$ and $h$ such that $R=\frac{|h^{-1}g_0|}{4}>0$.
Let $u(g)=\Gamma^L(g, h, \tau)$. Then $\,u\,$ satisfies the
equation $-\Delta_{\mathbb H^n}u+(V+i\tau)u=0$ in $B(g_0, 2R)$.
Take $\phi\in C^\infty_c(B(g_0, 2R))$ such that $\phi \equiv 1$ on
$B(g_0, R)$, $0 \leq \phi \leq 1$, $|\nabla_{\mathbb H^n} \phi|
\leq \frac{C}{R}$, and $|\nabla^2_{\mathbb H^n} \phi| \leq
\frac{C}{R^2}$. For $g'\in B(g_0, R)$, we have
\begin{eqnarray*}
u(g')&=& \int_{\mathbb H^n}\Gamma(g', g,
\tau)(-\Delta_{\mathbb H^n}+i\tau) (u\phi)(g)\, dg\\
&=&\int_{\mathbb H^n}\Gamma(g', g, \tau) \big( -V(g)
u(g)\phi(g)-2\nabla_{\mathbb H^n}u(g)\cdot\nabla_{\mathbb
H^n}\phi(g)-u(g) \Delta_{\mathbb H^n}\phi(g)\big)\, dg\\
&=&\int_{\mathbb H^n}\Gamma(g', g, \tau) \big( -V(g)u(g)\phi(g)
+ u(g)\Delta_{\mathbb H^n}\phi(g)\big )\, dg\\
& & +\,2\int_{\mathbb H^n} u(g)\nabla_{\mathbb
H^n,g} \Gamma(g', g, \tau)\cdot\nabla_{\mathbb H^n}\phi(g)\, dg.
\end{eqnarray*}
By $(\ref{a4})$,
\begin{eqnarray*}
\big| \nabla_{\mathbb H^n}u(g_0) \big| &\leq& C\int_{B(g_0, 2R)}
\frac{V(g)|u(g)|}{|g^{-1}_0 g|^{Q-1}}\, dg
+ \frac{C}{R^{Q+1}} \int_{B(g_0, 2R)}|u(g)|\, dg\\
&\leq& C \sup_{B(g_0, 2R)}|u(g)| \bigg( \int_{B(g_0, 2R)}
\frac{V(g)\, dg}{|g^{-1}_0 g|^{Q-1}} + \frac{1}{R} \bigg).
\end{eqnarray*}
Since $|h^{-1}g| \sim R$ for any $g \in B(g_0, 2R)$, by $(\ref{a5})$,
\begin{equation*}
\big| \nabla_{\mathbb H^n, g}\Gamma^L (g_0, h, \tau) \big| \leq
\frac{C_l}{(1+R\, |\tau|^{\frac{1}{2}})^l \big ( 1+R\,
\rho(h)^{-1} \big )^l} \bigg( \frac{1}{R^{Q-2}}\int_{B(g_0, 2R)}
\frac{V(g)\, dg}{|g^{-1}_0 g|^{Q-1}} + \frac{1}{R^{Q-1}}\bigg).
\end{equation*}
In view of (\ref{a10}), by integration we get
\begin{equation*}
\big| \mathcal{R}^L(g_0, h) \big| \leq \frac{C_l}{\big (1+R\,
\rho(h)^{-1} \big )^l}\bigg( \frac{1}{R^{Q-1}}\int_{B(g_0, 2R)}
\frac{V(g)\, dg}{|g^{-1}_0 g|^{Q-1}}+\frac{1}{R^{Q}}\bigg),
\end{equation*}
which means
\begin{equation}\label{a12}
\big| R^L_j(g,h) \big|
\leq \frac{C_l}{\big (1+ |h^{-1}g|\,
\rho(h)^{-1} \big )^l} \bigg( \frac{1}{|h^{-1}g|^{Q-1}}
\int_{B \big( g,\frac{\,|h^{-1}g|}{2} \big)}\frac{V(w)\, dw}
{|g^{-1}w|^{Q-1}} + \frac{1}{|h^{-1}g|^{Q}} \bigg).
\end{equation}

Set $r= A\, \rho(h)$. By $(\ref{a12})$ and the boundedness of
fractional integrals (cf. \cite[Proposition 6.2]{Folland-Stein}),
for $k\geq 1$ and $\frac{1}{p_{_0}}=\frac{1}{q_{_0}}-\frac{1}{Q}$, we get
\begin{eqnarray*}
&&\bigg( \int_{2^{k-1}r \leq |h^{-1}g|<2^k r} \big | R^L_j(g, h)
\big |^{p_{_0}}\, dg \bigg)^{\frac{1}{p_{_0}}}\\
&&\qquad \leq C_l\, 2^{-kl} \left ( \frac{1}{(2^k r)^{Q-1}} \bigg(
\int_{|h^{-1}g|<2^{k+1} r} V(g)^{q_{_0}}\, dg \bigg
)^{\frac{1}{q_{_0}}}+ (2^kr)^{\frac{Q}{p_{_0}}-Q}\right )\\
&&\qquad \leq C_l\, 2^{-kl}\, (2^kr)^{\frac{Q}{p_{_0}}-Q} \bigg(
\frac{1}{(2^{k+1} r)^{Q-2}} \int_{|h^{-1}
g|<2^{k+1} r} V(g)\, dg + 1 \bigg)\\
&&\qquad \leq C_l\, 2^{-kl}\, (2^kr)^{\frac{Q}{p_{_0}}-Q} \big (
2^{k\,l_1} + 1 \big ) \leq C\, 2^{-k}\, (2^k r)^{-\frac{Q}{p'_{_0}}}
\end{eqnarray*}
provided $l$ large enough, where we used the $B_{q_{_0}}$
condition for the second inequality and Lemma \ref{lem5} for the
third inequality .

Then by H\"{o}lder's inequality,
\begin{eqnarray*}
\int_{|h^{-1}g|\geq r} \big| R^L_j(g, h) \big|\, dg
&\leq& C\, \sum^\infty_{k=1} \bigg( \int_{2^{k-1}r \leq |h^{-1}g|<2^k r}
\big| R^L_j(g, h) \big|^{p_{_0}}\, dg
\bigg)^{\frac{1}{p_{_0}}} (2^k r)^{\frac{Q}{p'_{_0}}}\\
&\leq & C\, \sum^\infty_{k=1} 2^{-k} = C;
\end{eqnarray*}
the lemma is proved.
\end{proof}

\begin{lemma}\label{lem9} Let $A>0$ be a fixed constant. Then
\begin{eqnarray*}
\int_{B(h, A \rho(h))} \big| R^L_j(g, h)- R_j(g, h) \big|\, dg
\leq C .
\end{eqnarray*}
\end{lemma}

\begin{proof}
Set $r= A\, \rho(h)$. Suppose $|h^{-1}g| \leq r$ and let
$R=\frac{|h^{-1}g|}{4}$. By $(\ref{a4})$, $(\ref{a5})$ and $(\ref{a6})$,
\begin{eqnarray*}
&& \big| \nabla_{\mathbb{H}^n, g} \Gamma^L(g, h,\tau)-
\nabla_{\mathbb{H}^n, g} \Gamma(g, h,\tau) \big|\\
&&\qquad \leq \int_{\mathbb{H}^n} \big| \nabla_{\mathbb{H}^n,
g} \Gamma(g, w,\tau) \big|\, V(w)\, \big| \Gamma^L(w, h,\tau) \big|\, dw\\
&&\qquad \leq \int_{\mathbb{H}^n} \frac{C_l} {\big (
1+|g^{-1}w|\, |\tau|^{\frac{1}{2}} \big )^l |g^{-1}w|^{Q-1}} \\
&&\qquad \qquad \quad \cdot \frac{V(w)\, dw} {\big( 1+
|h^{-1}w|\, |\tau|^{\frac{1}{2}} \big)^l
\big( 1+ |h^{-1}w|\, \rho(h)^{-1} \big)^l |h^{-1}w|^{Q-2}}\\
&&\qquad = \int_{|g^{-1}w|<R}+ \int_{|h^{-1}w| <R}+
\int_{|h^{-1}w| \geq R,\, |g^{-1}w| \geq R }\\
&&\qquad = J_1+J_2+J_3.
\end{eqnarray*}
It is easy to see that
\begin{equation*}
J_1\ \ \leq \ \ \frac{C_l}{\big( 1+R\,
|\tau|^{\frac{1}{2}}\big)^l} \frac{1}{R^{Q-2}} \int_{B(g, R)}
\frac{V(w)\, dw}{|g^{-1}w|^{Q-1}}
\end{equation*}
and
\begin{eqnarray*}
J_2 &\leq& \frac{C_l} {\big( 1+R\, |\tau|^{\frac{1}{2}} \big)^l}
\frac{1}{R^{Q-1}} \int_{B(h, R)} \frac{V(w)\, dw}{|h^{-1}w|^{Q-2}}\\
&\leq& \frac{C_l} {\big( 1+R\, |\tau|^{\frac{1}{2}} \big)^l}
\frac{1}{R^{2Q-3}} \int_{B(h, R)} V(w)\, dw\\
&\leq& \frac{C_l} {\big(
1+R\, |\tau|^{\frac{1}{2}} \big)^l} \frac{1}{R^{Q-1}} \bigg(
\frac{R}{r} \bigg)^{2-\frac{Q}{q_{_0}}},
\end{eqnarray*}
where we used Lemma \ref{lem5} for the second inequality and
Lemma \ref{lem2} together with Lemma \ref{lem3} for the last
inequality.

Note that $|g^{-1}w| \sim |h^{-1}w|$ when $|g^{-1}w| \geq R$ and
$|h^{-1}w|\geq R$. We have
\begin{eqnarray*}
J_3 &\leq& \frac{C_l}{\big( 1+ R\, |\tau|^{\frac{1}{2}} \big)^l}
\int_{|h^{-1}w|\geq R} \frac{V(w)\,
dw}{\big( 1+|h^{-1}w|\, \rho(h)^{-1} \big)^l |h^{-1}w|^{2Q-3}}\\
&\leq& \frac{C_l}{\big( 1+ R\, |\tau|^{\frac{1}{2}} \big)^l}
\bigg( \int_{r>|h^{-1}w|\geq R} \frac{V(w)\, dw}{|h^{-1}w|^{2Q-3}}
+ r^l \int_{|h^{-1}w|\geq r} \frac{V(w)\, dw}{|h^{-1}w|^{2Q-3+l}}
\bigg).
\end{eqnarray*}
It follows from the H\"{o}lder inequality and $B_{q_{_0}}$
condition that
\begin{eqnarray*}
\int_{r>|h^{-1}w|\geq R} \frac{V(w)\, dw}{|h^{-1}w|^{2Q-3}} &\leq&
C \bigg( \int_{B(h,r)} V(w)^{q_{_0}}\, dw \bigg)^\frac{1}{q_{_0}}
\bigg( \int_R^r t^{(-2Q+3)q'_{_0}
+Q-1}\, dt \bigg)^\frac{1}{q'_{_0}}\\
&\leq& C\, r^{\frac{Q}{q_{_0}}-2} R^{-2Q+3+ \frac{Q}{q'_{_0}}}\\
&=& \frac{C}{R^{Q-1}} \bigg ( \frac{R}{r}\bigg )^{2-\frac{Q}{q_{_0}}}.
\end{eqnarray*}
Using Lemma \ref{lem5} and taking $l$ sufficiently large, we
obtain
\begin{eqnarray*}
r^l \int_{|h^{-1}w|\geq r} \frac{V(w)\, dw}{|h^{-1}w|^{2Q-3+l}}
&\leq& C\, r^l \sum_{j=1}^{\infty}
(2^j r)^{-2Q+3-l} \int_{B(h, 2^jr)} V(w)\, dw\\
&\leq& \frac{C}{r^{Q-1}} \sum_{j=1}^{\infty} 2^{-(l-l_1+Q-1)j}\\
&\leq& \frac{C}{R^{Q-1}} \bigg( \frac{R}{r}
\bigg)^{2-\frac{Q}{q_{_0}}}.
\end{eqnarray*}
The above three estimates give
\begin{equation*}
J_3 \leq \frac{C_l} {\big( 1+R\, |\tau|^{\frac{1}{2}} \big)^l}
\frac{1}{R^{Q-1}} \bigg ( \frac{R}{r}\bigg )^{2-\frac{Q}{q_{_0}}}.
\end{equation*}
Therefore, for $|h^{-1}g|<r \sim \rho(h)$,
\begin{eqnarray*}
&& \big|\nabla_{\mathbb{H}^n, g}\Gamma^L(g, h,\tau)-
\nabla_{\mathbb{H}^n, g}\Gamma(g, h,\tau) \big|\\
&& \leq \frac{C_l}{\big( 1+ |h^{-1}g|\, |\tau|^{\frac{1}{2}}
\big)^l} \left( \frac{1}{|h^{-1}g|^{Q-2}} \int_{B \big( g,
\frac{|h^{-1}g|}{4} \big)} \frac{V(w)\, dw}{|g^{-1}w|^{Q-1}} +
\frac{1}{|h^{-1}g|^{Q-1}} \bigg( \frac{|h^{-1}g|}{r}
\bigg)^{2-\frac{Q}{q_{_0}}} \right).
\end{eqnarray*}
In view of $(\ref{a10})$ and $(\ref{a11})$, by integration we
obtain
\begin{equation}\label{a13}
\big| R^L_j(g, h)- R_j(g, h) \big| \leq \frac{C}{|h^{-1}g|^{Q-1}}
\int_{B \big( g, \frac{|h^{-1}g|}{4} \big)} \frac{V(w)\,
dw}{|g^{-1}w|^{Q-1}} + \frac{C}{|h^{-1}g|^{Q}} \bigg (
\frac{|h^{-1}g|}{r}\bigg )^{2-\frac{Q}{q_{_0}}}.
\end{equation}
It follows from the boundedness of fractional integrals and
(\ref{a13}) that, for $k\leq 0$,
\begin{eqnarray*}
&& \bigg( \int_{2^{k-1}r \leq |h^{-1} g|<2^k r} \big| R^L_j(g, h)- R_j(g, h)
\big|^{p_{_0}}\, dg \bigg)^{\frac{1}{p_{_0}}}\\
&&\qquad \leq \frac{C}{(2^k r)^{Q-1}} \bigg( \int_{B(h, 2r)}
V(g)^{q_{_0}}\, dg \bigg)^{\frac{1}{q_{_0}}} + C\,
2^{k(2-\frac{Q}{q_{_0}})}(2^k r)^{\frac{Q}{p_{_0}}-Q}\\
&&\qquad \leq C\,
2^{k(2-\frac{Q}{q_{_0}})}(2^k r)^{-\frac{Q}{p'_{_0}}}.
\end{eqnarray*}
By H\"{o}lder's inequality, we obtain
\begin{eqnarray*}
&&\int_{|g^{-1}h|<r} \big| R^L_j(g, h)- R_j(g, h) \big|\, dg\\
&&\qquad \leq \sum^0_{k=-\infty} \bigg( \int_{2^{k-1}r \leq
|g^{-1}_0 h|< 2^k r} \big| R^L_j(g, h)- R_j(g, h) \big|^{p_{_0}}\,
dg \bigg)^{\frac{1}{p_{_0}}} (2^k r)^{\frac{Q}{p'_{_0}}}\\
&&\qquad \leq C\, \sum^0_{k=-\infty} 2^{k(2-\frac{Q}{q_{_0}})} =
C.
\end{eqnarray*}
The proof of Lemma \ref{lem9} is completed.
\end{proof}


\section {Local Hardy spaces}

An $H^1_L$-function is locally equal to a function
in a certain scaled local Hardy space. Before dealing with
$H^1_L(\mathbb{H}^n)$, we first discuss local Hardy
spaces on the Heisenberg group. The theory of local Hardy
spaces on $\mathbb{R}^n$ was studied by Goldberg \cite{Goldberg}.
It is sure that the local version of Hardy spaces can be extended
to more general setting such as homogeneous groups.
However, there is no reference about it as far as the authors know.

We recall some results of Hardy spaces on the Heisenberg
group (cf. \cite{Folland-Stein}).
Let $\mathscr {S}(\mathbb{H}^n)$ denote the Schwartz class and
$\mathscr {S}'(\mathbb{H}^n)$ be the space of tempered
distributions. For $f \in \mathscr {S}'(\mathbb{H}^n)$ and $\phi
\in \mathscr {S}(\mathbb{H}^n)$, we define the nontangential
maximal function $M_{\phi}f$ and the radial maximal function
$M^+_{\phi}f$ of $f$ with respect to $\phi$ by
\begin{eqnarray*}
M_{\phi}f(g) &=& \sup _{|g^{-1}h| < r < \infty} \big| f \ast \phi_r (h) \big|,\\
M^+_{\phi}f(g) &=& \sup _{r>0} \big| f \ast \phi_r (g) \big|,
\end{eqnarray*}
where
\begin{eqnarray*}
\phi_r (g)= r^{-Q} \phi \big( \delta_{\frac{1}{r}}g \big).
\end{eqnarray*}

Denote by $\{ Y_j:\; j=1, \cdots, 2n+1 \}$  the basis for the
right-invariant vector fields on $\mathbb H^n$ corresponding to
$\{ X_j:\; j=1, \cdots, 2n+1 \}$; that is,
\begin{eqnarray*}
Y_{2n+1}= X_{2n+1},\ \
Y_j= \frac{\partial}{\partial x_j}-2x_{n+j}
\frac{\partial}{\partial t},\ \ Y_{n+j}=\frac{\partial}{\partial
x_{n+j}}+2x_j \frac{\partial}{\partial t},\ \ j=1, \cdots, n.
\end{eqnarray*}
Let $N \in \mathbb{N}$ and $\phi \in \mathscr {S}(\mathbb{H}^n)$.
We define a seminorm on $\mathscr {S}(\mathbb{H}^n)$ by
\begin{eqnarray*}
\big\| \phi \big\|_{(N)} = \sup_{g \in \mathbb{H}^n \atop |I| \leq N} (1+
|g|)^{(N+1)(Q+1)} \big| Y^I \phi (g) \big|,
\end{eqnarray*}
where
\begin{eqnarray*}
Y=(Y_1, \cdots, Y_{2n+1}),\ \ I=(i_1, \cdots, i_{2n+1}),\ \
Y^I= Y^{i_1}_1 \cdots Y^{i_{2n+1}}_{2n+1},\ \ \text{and}\ \
|I|= \sum_{j=1}^{2n+1} i_j .
\end{eqnarray*}
For $f \in \mathscr {S}'(\mathbb{H}^n)$, we define the
nontangential grand maximal function $M_{(N)}f$ and the radial
grand maximal function $M^+_{(N)}f$ by
\begin{eqnarray*}
M_{(N)}f(g) &=& \sup _{\phi \in \mathscr {S} \atop \|\phi\|_{(N)} \leq 1} M_{\phi}f (g),\\
M^+_{(N)}f(g) &=& \sup _{\phi \in \mathscr {S} \atop \|\phi\|_{(N)}
\leq 1} M^+_{\phi}f (g).
\end{eqnarray*}
For $0<p \leq 1$, we set $N_p= [Q(\frac{1}{p} -1)] +1$. The Hardy
space $H^p(\mathbb{H}^n)$ is defined by
\begin{eqnarray*}
H^p(\mathbb{H}^n)= \big\{ f \in \mathscr {S}'(\mathbb{H}^n):\;
M_{(N_p)}f \in L^p(\mathbb{H}^n) \big\}
\end{eqnarray*}
with
\begin{eqnarray*}
\big\| f \big\|_{H^p} = \big\| M_{(N_p)}f \big\|_{L^p}.
\end{eqnarray*}

We say that a function $\phi \in \mathscr {S}(\mathbb{H}^n)$ is a
{\it commutative approximate identity} if $\phi$ satisfies
\begin{eqnarray*}
\int_{\mathbb{H}^n} \phi(g)\, dg =1\qquad \text{and}\qquad \phi_r
\ast \phi_s = \phi_s \ast \phi_r\quad \text{for all}\quad r,s >0.
\end{eqnarray*}
We note that $\phi_r \ast \phi_s = \phi_s \ast \phi_r$ if
$\phi(x,t)$ is radial or polyradial (cf. \cite{Folland-Stein})
with respect to $x$ .

\begin{proposition}[{\cite[Chapter 4]{Folland-Stein}}]\label{pro1}
Suppose $f \in \mathscr {S}'(\mathbb{H}^n)$ and $0<p \leq 1$. Let
$\phi$ be a commutative approximate identity and $N \geq N_p$
be fixed. Then
\begin{eqnarray*}
\big\| M_{(N)}f \big\|_{L^p} \sim \big\| M^+_{(N)}f \big\|_{L^p} \sim \big\| M_{\phi}f
\big\|_{L^p} \sim \big\| M^+_{\phi}f \big\|_{L^p}.
\end{eqnarray*}
\end{proposition}

Proposition \ref{pro1} tells us that all these maximal functions
give equivalent characterizations of the Hardy space
$H^p(\mathbb{H}^n)$. Specifically, the heat kernel $H_s(g)=
\varphi_{\!_{\sqrt{s}}}(g)$ where $\varphi(g) = H_1(g)$ is a
commutative approximate identity. Therefore the radial heat
maximal function $Mf= M^+_{\varphi}f$ characterizes the Hardy
space $H^p(\mathbb{H}^n)$.

A significant fact of the Hardy space $H^p(\mathbb{H}^n)$ is that
an $H^p$ distribution admits an atomic decomposition. Let $0<p
\leq 1 <q \leq \infty$. A function $a\in L^q(\mathbb{H}^n)$ is
called an $H^{p,q}$-atom if the following conditions hold:
\begin{itemize}
\item [\rm(i)] $\mathrm{supp}\, a \subset B(g_0,r),$
\item [\rm(ii)] $\big\| a \big\|_{L^{q}}
                \leq \big| B(g_0,r) \big|^{\frac{1}{q}-\frac{1}{p}},$
\item [\rm(iii)] $\displaystyle\int_{B(g_0,r)} a(g) g^I\, dg =0\ $ for $\ d(I) < N_p,$
\end{itemize}
where $g^I$ denotes the monomial
\begin{eqnarray*}
(x,t)^I= x^{i_1}_1 \cdots x^{i_{2n}}_{2n}t^{i_{2n+1}}_{2n+1}
\end{eqnarray*}
and $d(I)$ is the homogeneous degree of $g^I$ defined by
\begin{eqnarray*}
d(I)= 2 i_{2n+1} + \sum_{j=1}^{2n} i_j.
\end{eqnarray*}

\begin{proposition}[{\cite[Chapter 3]{Folland-Stein}}]\label{pro2}
Let $0<p \leq 1 <q \leq \infty$. A distribution
$f \in H^p(\mathbb{H}^n)$ if and only if $\,f$ can be written as
$f= \sum_j \lambda_j\, a_j$, where $a_j$ are $H^{p,q}$-atoms,
$\sum_j |\lambda_j|^p < \infty$, and the sum converges in the
sense of distributions. Moreover,
\begin{eqnarray*}
\big\| f \big\|^p_{H^p} \sim \inf \bigg \{ \sum_j |\lambda_j|^p \bigg \},
\end{eqnarray*}
where the infimum is taken over all atomic decompositions of $\,f$
into $H^{p,q}$-atoms.
\end{proposition}

Now we go to the local version of Hardy spaces. We define the
local maximal functions $\widetilde{M}f$'s by taking supremum over
$0<r \leq 1$ instead of $0<r< \infty$ as follows.
\begin{align*}
  &\widetilde{M}_{\phi}f(g)=\sup_{|g^{-1}h|<r\le 1} \big| f*\phi_r(h) \big|,
 &&\widetilde{M}^+_{\phi}f(g)=\sup_{0<r\le 1} \big| f*\phi_r(g) \big|,\\
  &\widetilde{M}_{(N)}f(g)=\sup_{\phi\in\mathscr S \atop \|\phi\|_{(N)}\le 1}
    \widetilde{M}_{\phi}f(g),
 &&\widetilde{M}^+_{(N)}f(g)=\sup_{\phi\in\mathscr S \atop \|\phi\|_{(N)}\le 1}
    \widetilde{M}^+_{\phi}f(g).
\end{align*}

\begin{definition}\label{def3}
{\rm Let $0<p \leq 1$. The local Hardy space $h^p(\mathbb{H}^n)$ is
defined by
\begin{eqnarray*}
h^p(\mathbb{H}^n)= \big\{ f \in \mathscr {S}'(\mathbb{H}^n):\;
\widetilde{M}_{(N_p)}f \in L^p(\mathbb{H}^n) \big\}
\end{eqnarray*}
with
\begin{eqnarray*}
\big\| f \big\|_{h^p} = \big\| \widetilde{M}_{(N_p)}f \big\|_{L^p}.
\end{eqnarray*}}
\end{definition}

Similar to Proposition \ref{pro1}, by the same argument as in
\cite{Folland-Stein}, we have

\begin{proposition}\label{pro3}
Suppose $f \in \mathscr {S}'(\mathbb{H}^n)$ and $\,0<p \leq 1$. Let
$\phi$ be a commutative approximate identity and $N \geq N_p$
be fixed. Then
\begin{eqnarray*}
\big\| \widetilde{M}_{(N)}f \big\|_{L^p} \sim \big\| \widetilde{M}^+_{(N)}f
\big\|_{L^p} \sim \big\| \widetilde{M}_{\phi}f \big\|_{L^p} \sim \big\|
\widetilde{M}^+_{\phi}f \big\|_{L^p}.
\end{eqnarray*}
\end{proposition}

All these local maximal functions give equivalent
characterizations of the local Hardy space $h^p(\mathbb{H}^n)$.

\begin{lemma}\label{lem10} Let $0<p \leq 1$ and $f \in
h^p(\mathbb{H}^n)$. If $\psi \in \mathscr {S}(\mathbb{H}^n)$
satisfies
\begin{eqnarray*}
\int_{\mathbb{H}^n} \psi(g)\, dg =1 \quad \text{and} \quad
\int_{\mathbb{H}^n} \psi(g) g^I\, dg =0 \quad \text{for} \quad
0<d(I)< m= (N_p+1)(Q+2)
\end{eqnarray*}
(the cancellation conditions of $\psi$ can be relaxed to an
extent, but it is not important for us), then $f - f \ast \psi \in
H^p(\mathbb{H}^n)$ and there exists $C>0$ such that
\begin{eqnarray*}
\big\| f - f \ast \psi \big\|_{H^p} \leq C\, \big\| f \big\|_{h^p}.
\end{eqnarray*}
\end{lemma}

\begin{proof}
Let $\phi$ be a commutative approximate identity. We have
\begin{eqnarray}
M^+_{\phi}(f - f \ast \psi)(g)
&\leq& \widetilde{M}^+_{\phi}f(g)
      + \sup_{0<r \leq 1}\big| f \ast (\psi \ast \phi_r) (g) \big| \nonumber\\
& & +\, \sup_{1<r< \infty} \big| f \ast (\phi_r - \psi \ast \phi_r) (g) \big|. \label{a14}
\end{eqnarray}
We will prove that there exists $C>0$ such that
\begin{eqnarray}
\sup _{0<r \leq 1} \big| f \ast (\psi \ast \phi_r) (g) \big|
&\leq& C\, \widetilde{M}^+_{(N_p)}f (g),\label{a15}\\
\sup _{1<r< \infty } \big | f \ast (\phi_r - \psi \ast \phi_r) (g)
\big | &\leq& C\, \widetilde{M}^+_{(N_p)}f (g). \label{a16}
\end{eqnarray}
Then Lemma \ref{lem10} follows from $(\ref{a14})$, $(\ref{a15})$,
$(\ref{a16})$ and Proposition \ref{pro3}.

By \cite[Proposition 1.47]{Folland-Stein}, the estimate $(\ref{a15})$ follows from
\begin{eqnarray*} \sup_{0<r \leq 1} \big\|
\psi \ast \phi_r \big\|_{(N_p)} \leq C\, \big\| \psi \big\|_{(N_p +1)} \big\| \phi
\big\|_{(N_p +1)}.
\end{eqnarray*}
Note that
\begin{eqnarray*}
Y^I= \sum_{|J| \leq |I|\atop d(J) \geq d(I)} P_{I\!J} X^J,
\end{eqnarray*}
where $P_{I\!J}$ are homogeneous polynomials of degree
$d(J)-d(I)$. Hence
\begin{eqnarray*}
\big\| \phi \big\|_{(N)} \leq C \sup_{g \in \mathbb{H}^n\atop |I| \leq N}
(1+ |g|)^{(N+1)(Q+2)} \big| X^I \phi (g) \big|.
\end{eqnarray*}
To obtain the estimate $(\ref{a16})$, we only need to show
\begin{eqnarray}\label{a17}
\sup_{1<r < \infty\atop |I| \leq N_p} \big | X^I (\phi_r - \psi \ast
\phi_r) (g) \big| \leq C\, (1+ |g|)^{-m}.
\end{eqnarray}
For $f= X^I \phi$, we use the following Taylor inequality
(cf. \cite[Corollary 1.44]{Folland-Stein})
\begin{eqnarray*}
\big| f(hg) - P_g(h) \big| \leq C\, |h|^{m} \sup _{|h_1| \leq
b^{m} |h|\atop d(J)=m} \big| Y^Jf(h_1g) \big|,
\end{eqnarray*}
where $P_g(h)$ is the right Taylor polynomial of $f$ at the point
$g$ of degree $m-1$ and $b>1$ is a constant.
From the cancellation conditions of $\psi$ we get
\begin{eqnarray*}
\big| X^I (\phi_r - \psi \ast \phi_r) (g) \big|
&=& \bigg| \int_{\mathbb{H}^n} \psi (h) r^{- d(I)} \big( (X^I
\phi)_r (g) - (X^I \phi)_r (h^{-1}g) \big)\, dh \bigg|\\
&\leq& C \int_{\mathbb{H}^n} \big| \psi (h) \big|\, |h|^{m}\, r^{-d(I)-m}
\sup _{|h_1| \leq b^{m} |h|\atop d(J)=m} \big| (Y^JX^I
\phi)_r (h_1g) \big|\, dh\\
&\leq& C \bigg ( \int_{|h| \leq \frac{1}{2} b^{-m}|g|} +
\int_{|h| > \frac{1}{2} b^{-m} |g|} \bigg ).
\end{eqnarray*}
For $|h| \leq \frac{1}{2} b^{-m}|g|$,
$$r^{- d(I) -m} \sup_{|h_1| \leq b^{m} |h|\atop d(J)=m} \big|
  (Y^JX^I \phi)_r(h_1g) \big|
  \leq C\, (1+ |g|)^{-d(I)-m-Q}  \leq C\, (1+ |g|)^{-m}.$$
Therefore
\begin{eqnarray*}
&&\int_{|h| \leq \frac{1}{2} b^{-m}|g|} \big| \psi(h) \big|\, |h|^{m}\,
r^{-d(I)-m} \sup _{|h_1| \leq b^{m} |h|\atop d(J)=m} \big| (Y^JX^I
\phi)_r (h_1g) \big|\, dh \\
&&\qquad \leq C\, (1+ |g|)^{-m} \int_{|h| \leq \frac{1}{2} b^{-m}|g|}
\big| \psi(h) \big|\, |h|^{m}\, dh\\
&&\qquad\leq C\, (1+ |g|)^{-m}.
\end{eqnarray*}
On the other hand,
\begin{eqnarray*}
&&\int_{|h|> \frac{1}{2} b^{-m} |g|} \big| \psi(h) \big|\, |h|^{m}\,
r^{-d(I)-m} \sup _{|h_1| \leq b^{m} |h|\atop d(J)=m} \big| (Y^JX^I
\phi)_r (h_1g) \big|\, dh \\
&&\qquad\leq C \int_{|h| > \frac{1}{2} b^{-m} |g|} \big| \psi (h) \big|\,|h|^{m}\, dh\\
&&\qquad\leq C\, (1+ |g|)^{-m}.
\end{eqnarray*}
Thus we obtain $(\ref{a17})$ and complete the proof of Lemma
\ref{lem10}.
\end{proof}

Although the ball $B(g,r)$ is the left translation by $g$ of
$B(0,r)$, the shape of $B(g,r)$, as a set of points in
$\mathbb{R}^{2n+1}$, much varies with the position of the center
$g$. For $g=(x,t)$ and $h=(y,s)$, let $d_E(g,h)= \big( |x-y|^2+|t-s|^2
\big)^{\frac{1}{2}}$ denote the Euclidean distance between two points
$g$ and $h$. Given a large positive number $A$, there exist points
$g$ and $h$ such that $d(g,h)=1$ and $d_E(g,h) \geq A$, and vice
versa. However, we still have the following covering lemma for $\mathbb{H}^n$.

\begin{lemma}\label{lem11}
Let $\alpha=(\alpha_1, \alpha_2, \cdots, \alpha_{2n+1})\in \mathbb{Z}^{2n+1}$,
$g_{\alpha}= \delta_{\frac{1}{2 \sqrt{Q}}}(\alpha_1, \alpha_2, \cdots, \alpha_{2n+1})$,
and $B_{\alpha} = B(g_{\alpha}, \frac{1}{2})$. 
Then $\mathbb{H}^n= \bigcup_{\alpha}\, B_{\alpha}$ and
$\{ B_{\alpha} \}$ has finite overlaps property; that is,
$1 \leq \sum_{\alpha} \chi_{_{B_\alpha}} (g) \leq C$.
\end{lemma}

\begin{proof}
Lemma \ref{lem11} is almost obvious. In fact, for any $g \in
\mathbb{H}^n$, there exists some $g_{\alpha}$ such that
$|g^{-1}_{\alpha}g| < \frac{1}{2}$. On the other hand, for any
$g_{\alpha} \neq g_{\alpha'}$, $|g^{-1}_{\alpha} g_{\alpha'}| \geq
\frac{1}{2 \sqrt{Q}}$.
\end{proof}

Now we give the atomic decomposition of $h^p(\mathbb{H}^n)$.

\begin{definition}\label{def4}
{\rm Let $0<p \leq 1 <q \leq \infty$. A function $a\in L^q(\mathbb{H}^n)$ is called an
$h^{p,q}$-atom if the following conditions hold:
\begin{itemize}
\item [\rm(i)] $\mathrm{supp}\, a \subset B(g_0,r),$
\item [\rm(ii)] $\big\| a \big\|_{L^{q}} \leq \big| B(g_0,r) \big|^{\frac{1}{q}-\frac{1}{p}},$
\item [\rm(iii)] if $\ r < \frac{1}{2},$
                 then $\displaystyle\int_{B(g_0,r)} a(g) g^I\, dg =0\ $ for $\ d(I) < N_p.$
\end{itemize}}
\end{definition}

\begin{theorem}\label{thm4}
Let $0<p \leq 1 <q \leq \infty$. A distribution
$f \in h^p(\mathbb{H}^n)$ if and only if $\,f$ can be written as
$f= \sum_j \lambda_j\, a_j$ converging in the sense of
distributions and in $h^p(\mathbb{H}^n)$ norm, where $a_j$ are $h^{p,q}$-atoms and
$\sum_j |\lambda_j|^p < \infty$. Moreover,
\begin{eqnarray*}
\big\| f \big\|^p_{h^p} \sim \inf \bigg \{ \sum_j |\lambda_j|^p \bigg \},
\end{eqnarray*}
where the infimum is taken over all atomic decompositions of $\,f$
into $h^{p,q}$-atoms.
\end{theorem}

\begin{proof}
Let $f \in h^p(\mathbb{H}^n)$ and $\psi$ satisfy the conditions in 
Lemma \ref{lem10}. Thus, $f - f \ast \psi \in H^p(\mathbb{H}^n)$ and
$\| f - f \ast \psi \|_{H^p} \leq C\, \|f\|_{h^p}$. 
By the atomic decomposition of $H^p(\mathbb{H}^n)$, 
$f - f \ast \psi = \sum_j \lambda_j\, a_j$, where $a_j$ are $H^{p,q}$-atoms
and $\sum_j |\lambda_j|^p \le C\, \| f - f \ast \psi \|_{H^p}^p$. 
It is clear that an $H^{p,q}$-atom is also an $h^{p,q}$-atom.

Let $\{ B_{\alpha}\}$ be the collection of balls given in Lemma
\ref{lem11} and $\chi_{\alpha}$ denote the characteristic function
of $B_{\alpha}$. Set
\begin{eqnarray}\label{a18}
\xi_{\alpha}(g)= \frac{\chi_{\alpha}(g)}{\sum_{\alpha}
\chi_{\alpha}(g)}
\end{eqnarray}
and write
\begin{eqnarray*}
(f \ast \psi)(g)\, \xi_{\alpha}(g)= \lambda_{\alpha}\,a_{\alpha}(g),
\end{eqnarray*}
where
\begin{eqnarray*}
\lambda_{\alpha}= |B_{\alpha}|^{\frac 1p - \frac 1q}\, \big\| (f \ast \psi)
\xi_{\alpha} \big\|_{L^q} .
\end{eqnarray*}
Then all $a_{\alpha}$ are $h^{p,q}$-atoms and $f \ast \psi =
\sum_{\alpha} \lambda_{\alpha}\, a_{\alpha}$. Moreover,
\begin{eqnarray*}
\sum_{\alpha} |\lambda_{\alpha}|^p = \sum_{\alpha} |B_{\alpha}|^{1-\frac pq} \big\|
(f \ast \psi) \xi_{\alpha} \big\|^p_{L^q} \leq \sum_{\alpha}
\int_{B_{\alpha}} \big ( \widetilde{M}_{\psi}f (g) \big )^p\, dg
\leq  C\, \big\| f \big\|^p_{h^p}.
\end{eqnarray*}

Conversely, let $a$ be an $h^{p,q}$-atom supported on a ball
$B(g_0, r)$. If $r < \frac{1}{2}$, then $a$ is also an
$H^{p,q}$-atom and $\big\| a \big\|^p_{h^p} \leq \big\| a
\big\|^p_{H^p} \leq C$ by Proposition \ref{pro2}. If $r \geq
\frac{1}{2}$, we choose a commutative approximate identity $\phi$
such that $\text{supp}\, \phi \subset B(0, \frac{1}{2})$. Then
$\widetilde{M}^+_{\phi}a$ is supported in $B(g_0, 2r)$ and $\big\|
\widetilde{M}^+_{\phi}a \big\|_{L^{q}} \leq C\, \big\| a
\big\|_{L^{q}}$. By Proposition \ref{pro3},
\begin{eqnarray*}
\big\| a \big\|^p_{h^p} \leq C\, \big\| \widetilde{M}^+_{\phi}a
\big\|^p_{L^p} \leq C\, \big| B(g_0, 2r) \big|^{1-\frac{p}{q}}\,
\big\| \widetilde{M}^+_{\phi}a \big\|^p_{L^q} \leq C\, \big|
B(g_0, r) \big|^{1-\frac{p}{q}}\, \big\| a \big\|^p_{L^{q}} \leq
C\, .
\end{eqnarray*}
If $f= \sum_j \lambda_j\, a_j$, then we have
\begin{eqnarray*}
\big\| f \big\|^p_{h^p} \leq \sum_j |\lambda_j|^p \big\| a_j \big\|^p_{h^p}
\leq C  \sum_j |\lambda_j|^p.
\end{eqnarray*}
The proof of Theorem \ref{thm4} is completed.
\end{proof}

\begin{remark}\label{rem1}
{\rm The restriction $r< \frac{1}{2}$ in the condition (iii) of
Definition \ref{def4} is not essential. By a dilation argument, it
is easy to know that Theorem \ref{thm4} still holds if we replace
$\frac{1}{2}$ by a fixed positive constant $A$.}
\end{remark}

\begin{remark}\label{rem2}
{\rm We should pay attention to the $h^{p,q}$-atoms appearing in the
atomic decomposition of $\,f\in h^p(\mathbb{H}^n)$ without the cancellation
conditions, which come from $f \ast \psi$. 
If $\,f \in h^p(\mathbb{H}^n)$ is supported in a ball $B(g_0, 1)$, 
we can choose $\psi$ satisfying the conditions in Lemma \ref{lem10} and
$\mathrm{supp}\, \psi \subset B(0, 1)$. Then $f \ast \psi$ itself
is multiple of $h^{p,q}$-atom supported on the ball $B(g_0, 2)$.}
\end{remark}

We are mainly interested in the case $p=1$. The Hardy space
$H^1(\mathbb{H}^n)$ can be characterized by the Riesz transforms $R_j$
(cf. \cite{Christ}). That is, $f \in H^1(\mathbb{H}^n)$ if and only if
$f \in L^1(\mathbb{H}^n)$ and $R_jf \in L^1(\mathbb{H}^n),\, j=1,
\cdots, 2n$. Moreover,
\begin{eqnarray}\label{a19}
\big\| f \big\|_{H^1} \sim \big\| f \big\|_{L^1} + \sum_{j=1}^{2n} \big\| R_jf
\big\|_{L^1}.
\end{eqnarray}
The local Hardy space $h^1(\mathbb{H}^n)$ has a similar
characterization. On the Euclidean spaces the local Riesz
transforms characterization of $h^1(\mathbb{R}^n)$ is obtained via
subharmonicity (cf. \cite{Goldberg}). However, this kind of approach
fails on the Heisenberg group as pointed out in \cite{Christ}, so we
will use a different method. Let $\zeta \in
C^{\infty}(\mathbb{H}^n)$ satisfy $0 \leq \zeta (g) \leq 1$,
$\zeta (g)=1$ for $|g| < \frac{1}{2}$ and $\zeta (g)=0$ for $|g| >
1$. We also assume that $\zeta$ is radial; that is, $\zeta (g)$
depends only on $|g|$. Set $R^{[0]}_j(g)= \zeta (g)\, R_j(g)$ and
$R^{\infty}_j(g)= \big ( 1-\zeta (g) \big )\, R_j(g)$. Then we
define the local Riesz transforms by $\widetilde{R}_jf= f \ast
R^{[0]}_j,\, j=1, \cdots, 2n$.

\begin{theorem}\label{thm5}
A function $f \in h^1(\mathbb{H}^n)$ if and only if $f \in
L^1(\mathbb{H}^n)$ and $\widetilde{R}_jf \in L^1(\mathbb{H}^n),\,
j=1, \cdots, 2n$. Moreover,
\begin{eqnarray*}
\big\| f \big\|_{h^1} \sim \big\| f \big\|_{L^1} + \sum_{j=1}^{2n} \big\|
\widetilde{R}_jf \big\|_{L^1}.
\end{eqnarray*}
\end{theorem}

\begin{proof}
Suppose $f \in L^1(\mathbb{H}^n)$ and $\widetilde{R}_jf \in
L^1(\mathbb{H}^n),\, j=1, \cdots, 2n$. Let $\{ B_{\alpha}\} =
\big\{ B(g_{\alpha}, \frac{1}{2}) \big\}$ be the collection of balls
given in Lemma \ref{lem11}. Let $B^{**}_{\alpha}=B(g_{\alpha}, 2)$
and $\chi^{**}_{\alpha}$ is the characteristic function of
$B^{**}_{\alpha}$. Write
\begin{eqnarray*}
f_{\alpha}(g)= f(g) \xi_{\alpha} (g),
\end{eqnarray*}
where $\xi_{\alpha}$ is defined by $(\ref{a18})$. Also set
\begin{eqnarray*}
\widetilde{f}(g)= \sum_{\alpha} \lambda_{\alpha} a_{\alpha}(g),
\end{eqnarray*}
where
\begin{eqnarray*}
a_{\alpha}(g)= \frac{1}{| B^{**}_{\alpha} |}
\chi^{**}_{\alpha}(g)\quad \text{and}\quad \lambda_{\alpha}=
\int_{B_{\alpha}} f_{\alpha}(g)\, dg.
\end{eqnarray*}
It is obvious that all $a_{\alpha}$ are $h^{1, \infty}$-atoms and
\begin{eqnarray*}
\sum_{\alpha} |\lambda_{\alpha}| \leq \sum_{\alpha}
\int_{B_{\alpha}} \big| f(g) \big|\, dg \leq C\, \big\| f \big\|_{L^1}.
\end{eqnarray*}
By Theorem \ref{thm4}, $\widetilde{f} \in h^1(\mathbb{H}^n)$ and
\begin{eqnarray*}
\big\| \widetilde{f} \big\|_{h^1} \leq C\, \big\| f \big\|_{L^1}.
\end{eqnarray*}
We will prove $f- \widetilde{f} \in H^1(\mathbb{H}^n)$ and
\begin{eqnarray}\label{a20}
\big\| f- \widetilde{f} \big\|_{H^1} \leq C\, \bigg( \big\| f \big\|_{L^1} +
\sum_{j=1}^{2n} \big\| \widetilde{R}_jf \big\|_{L^1} \bigg),
\end{eqnarray}
which imply $f \in h^1(\mathbb{H}^n)$ and
\begin{eqnarray*}
\big\| f \big\|_{h^1} \leq \big\| f- \widetilde{f} \big\|_{H^1} + \big\| \widetilde{f}
\big\|_{h^1} \leq C\, \bigg( \big\| f \big\|_{L^1} + \sum_{j=1}^{2n} \big\|
\widetilde{R}_jf \big\|_{L^1} \bigg).
\end{eqnarray*}
Write
\begin{eqnarray}\label{a21}
R_j \big( f- \widetilde{f} \big)(g)= f \ast R^{[0]}_j (g) - \widetilde{f}
\ast R^{[0]}_j (g) + \big( f- \widetilde{f} \big) \ast R^{\infty}_j (g).
\end{eqnarray}
It is well known that $R_j(g)$ is a Calder\'on-Zygmund kernel
satisfying
\begin{itemize}
\item [\rm(a)] $\big| R_j(g) \big| \leq \frac{C}{\; |g|^{Q}},$
\item [\rm(b)] $\big| R_j(hg)- R_j(g) \big| \leq \frac{C\, |h|}{\; |g|^{Q+1}}\ $
                for $\ |h|< \frac{|g|}{2},$
\item [\rm(c)] $\displaystyle\int_{a<|g|<b} R_j(g)\, dg =0\ $ for any $\ 0<a<b< \infty.$
\end{itemize}
If $|g^{-1}_{\alpha} g| \leq 1$ or $|g^{-1}_{\alpha}g| \geq 3$,
then $a_{\alpha} \ast R^{[0]}_j (g)=0$. When $1<|g^{-1}_{\alpha}
g|<3$, it is not difficult to get
\begin{eqnarray*}
\big| a_{\alpha} \ast R^{[0]}_j (g) \big| \leq C\, \log \frac{1}{\big| 2-
|g^{-1}_{\alpha} g| \big|},
\end{eqnarray*}
and hence
\begin{eqnarray*}
\big\| a_{\alpha} \ast R^{[0]}_j \big\|_{L^1} \leq C.
\end{eqnarray*}
It follows that
\begin{eqnarray}\label{a22}
\big\| \widetilde{f} \ast R^{[0]}_j \big\|_{L^1} \leq \sum_{\alpha}
|\lambda_{\alpha}|\, \big\| a_{\alpha} \ast R^{[0]}_j \big\|_{L^1}  \leq C
\sum_{\alpha} |\lambda_{\alpha}| \leq C\, \big\| f \big\|_{L^1}.
\end{eqnarray}
Since $\big| R^{\infty}_j (g) \big| \leq C$, we have
\begin{eqnarray*}
\big| (f_{\alpha}- \lambda_{\alpha} a_{\alpha}) \ast R^{\infty}_j
(g) \big| \leq C\, \big\| f_{\alpha}- \lambda_{\alpha} a_{\alpha}
\big\|_{L^1} \leq C\, \big\| f_{\alpha} \big\|_{L^1}.
\end{eqnarray*}
Moreover, when $|g^{-1}_{\alpha} g| >4$, we have
\begin{eqnarray*}
&& \big| (f_{\alpha}- \lambda_{\alpha} a_{\alpha}) \ast
R^{\infty}_j (g) \big|\\
&&\qquad = \bigg| \int_{B^{**}_{\alpha}} (f_{\alpha}- \lambda_{\alpha}
a_{\alpha})(h)\, \big( R^{\infty}_j (h^{-1}g) - R^{\infty}_j
(g_\alpha^{-1}g) \big)\, dh \bigg| \\
&&\qquad \leq \int_{B^{**}_{\alpha}} \big| (f_{\alpha}- \lambda_{\alpha}
a_{\alpha})(h) \big|\, \big| R^{\infty}_j(h^{-1}g)- R^{\infty}_j
(g_\alpha^{-1}g) \big| \, dh\\
&&\qquad \leq \frac{C}{|g^{-1}_{\alpha}g|^{Q+1}} \int_{B^{**}_{\alpha}}
\big| (f_{\alpha}- \lambda_{\alpha} a_{\alpha})(h) \big|\, dh\\
&&\qquad\leq \frac{C}{|g^{-1}_{\alpha}g|^{Q+1}}\, \big\| f_{\alpha} \big\|_{L^1}
\end{eqnarray*}
that implies
\begin{eqnarray*}
\big\| (f_{\alpha}- \lambda_{\alpha} a_{\alpha}) \ast R^{\infty}_j
\big\|_{L^1} \leq C\, \big\| f_{\alpha} \big\|_{L^1}.
\end{eqnarray*}
Thus,
\begin{equation}\label{a23}
\big\| \big( f- \widetilde{f} \big) \ast R^{\infty}_j \big\|_{L^1} \leq
\sum_{\alpha} \big\| (f_{\alpha}- \lambda_{\alpha} a_{\alpha}) \ast
R^{\infty}_j \big\|_{L^1} \leq C \sum_{\alpha}  \big\| f_{\alpha} \big\|_{L^1}
\leq C\, \big\| f \big\|_{L^1}.
\end{equation}
From $(\ref{a21})$, $(\ref{a22})$ and $(\ref{a23})$ we obtain
\begin{eqnarray}\label{a24}
\big\| R_j \big( f- \widetilde{f} \big) \big\|_{L^1} \leq C\, \Big( \big\| f \big\|_{L^1} +
\big\| \widetilde{R}_jf \big\|_{L^1} \Big).
\end{eqnarray}
Then $(\ref{a20})$ follows from $(\ref{a19})$ and $(\ref{a24})$.

Conversely, suppose $f \in h^1(\mathbb{H}^n)$. We choose $\psi$
satisfying the conditions in Lemma \ref{lem10} and
$\mathrm{supp}\,\psi \subset B(0, 1)$.
Using $(\ref{a19})$ and Lemma \ref{lem10}, we get
\begin{eqnarray}\label{a25}
\big\| R_j(f - f \ast \psi ) \big\|_{L^1} \leq C\, \big\| f - f \ast \psi
\big\|_{H^1} \leq C\, \big\| f \big\|_{h^1}.
\end{eqnarray}
Note that
\begin{eqnarray}\label{a26}
\widetilde{R}_jf = R_j(f - f \ast \psi )+ f \ast \phi,
\end{eqnarray}
where
\begin{eqnarray*}
\phi = \psi \ast R_j - R^{\infty}_j \in C^{\infty}(\mathbb{H}^n).
\end{eqnarray*}
If $|g| >2$, by the same argument as $(\ref{a17})$, we have
\begin{eqnarray*}
\big| \phi (g) \big|= \big| \psi \ast R_j (g) - R^{\infty}_j (g) \big| \leq
C\, (1+ |g|)^{-Q-1}.
\end{eqnarray*}
Thus $\phi \in L^1(\mathbb{H}^n)$ and
\begin{eqnarray}\label{a27}
\big\| f \ast \phi \big\|_{L^1} \leq C\, \big\| f \big\|_{L^1} \leq C\, \big\| f
\big\|_{h^1}.
\end{eqnarray}
From $(\ref{a25})$, $(\ref{a26})$ and $(\ref{a27})$, we obtain
\begin{eqnarray*}
\big\| \widetilde{R}_jf \big\|_{L^1} \leq \big\| R_j(f - f \ast \psi )
\big\|_{L^1} + \big\| f \ast \phi \big\|_{L^1} \leq C\, \big\| f \big\|_{h^1}.
\end{eqnarray*}
The proof of Theorem \ref{thm5} is completed.
\end{proof}

\begin{remark}\label{rem3}
{\rm By a dilation argument, for $0<p \leq 1$ and $k\in \mathbb{Z}$, we
define the scaled local Hardy space $h^p_k(\mathbb{H}^n)$ to be
\begin{eqnarray*}
h^p_k(\mathbb{H}^n)= \big\{ f \in \mathscr {S}'(\mathbb{H}^n):\;
\widetilde{M}_{(N_p),\, k}f \in L^p(\mathbb{H}^n) \big\}
\end{eqnarray*}
with
\begin{eqnarray*}
\big\| f \big\|_{h^p_k} = \big\| \widetilde{M}_{(N_p),\, k}f \big\|_{L^p},
\end{eqnarray*}
where the scaled local maximal functions $\widetilde{M}_k f$'s are
defined by taking supremum over $\,0<r \leq 2^{k}$ instead of
$\,0<r\leq 1:$
\begin{align*}
&\widetilde{M}_{\phi,k}f(g)=\sup_{|g^{-1}h|<r\le 2^k} \big| f*\phi_r(h) \big|,
&&\widetilde{M}^+_{\phi,k}f(g)=\sup_{0<r\le 2^k} \big| f*\phi_r(g) \big|,\\
&\widetilde{M}_{(N),k}f(g)=\sup_{\phi\in\mathscr S \atop \|\phi\|_{(N)}\le 1}
\widetilde{M}_{\phi,k}f(g),
&&\widetilde{M}^+_{(N),k}f(g)=\sup_{\phi\in\mathscr S \atop \|\phi\|_{(N)}\le 1}
\widetilde{M}^+_{\phi,k}f(g).
\end{align*}
In a similar way, we have
\begin{eqnarray*}
\big\| \widetilde{M}_{(N),\, k}f \big\|_{L^p} \sim \big\|
\widetilde{M}^+_{(N),\, k}f \big\|_{L^p} \sim \big\|
\widetilde{M}_{\phi,\, k}f \big\|_{L^p} \sim \big\|
\widetilde{M}^+_{\phi,\, k}f \big\|_{L^p},
\end{eqnarray*}
where $\phi$ is a commutative approximate identity and $N \geq
N_p$ is fixed.

We also have the atomic decomposition of $h^p_k(\mathbb{H}^n)$ as follows.
Let $0<p \leq 1<q \leq \infty$. A function $a\in L^q(\mathbb{H}^n)$ is
called an $h^{p,q}_k$-atom if the following conditions hold:
\begin{itemize}
\item [\rm(i)] $\mathrm{supp}\, a \subset B(g_0,r),$
\item [\rm(ii)] $\big\| a \big\|_{L^{q}} \leq \big| B(g_0,r) \big|^{\frac{1}{q}-\frac{1}{p}},$
\item [\rm(iii)] if $\ r < 2^{k},$ then $\displaystyle\int_{B(g_0,r)} a(g) g^I\, dg =0\ $ for $\ d(I) < N_p.$
\end{itemize}
Then $f \in h^p_k(\mathbb{H}^n)$ if and only if $\,f$ can be
written as $f= \sum_j \lambda_j\, a_j$ converging in the sense of
distributions and in $h^p_k(\mathbb{H}^n)$ norm, where $a_j$ are
$h^{p,q}_k$-atoms and $\sum_j |\lambda_j|^p < \infty$. Moreover,
\begin{eqnarray*}
\big\| f \big\|^p_{h^p_k} \sim \inf \bigg \{ \sum_j |\lambda_j|^p \bigg \},
\end{eqnarray*}
where the infimum is taken over all atomic decompositions of $\,f$
into $h^{p,q}_k$-atoms. Furthermore, if $f \in
h^p_k(\mathbb{H}^n)$ is supported in a ball $B(g_0, 2^{k})$, then
the $h^{p,q}_k$-atom appearing in the atomic decomposition of $\,f$
without the cancellation condition is supported on the
ball $B(g_0, 2^{k+1})$.

In the case $p=1$, $h^1_k(\mathbb{H}^n)$ is also characterized by
the local Riesz transforms as follows. A function $f \in
h^1_k(\mathbb{H}^n)$ if and only if $f \in L^1(\mathbb{H}^n)$ and
$\widetilde{R}^{[k]}_jf \in L^1(\mathbb{H}^n),\, j=1, \cdots, 2n$,
where the local Riesz transforms are defined by
$\widetilde{R}^{[k]}_jf= f \ast R^{[k]}_j$ and $R^{[k]}_j(g)=
\zeta (2^{-k} g)\, R_j(g)$. Moreover,
\begin{eqnarray*}
\big\| f \big\|_{h^1_k} \sim \big\| f \big\|_{L^1} + \sum_{j=1}^{2n} \big\|
\widetilde{R}^{[k]}_jf \big\|_{L^1}.
\end{eqnarray*}}
\end{remark}


\section {Proof of atomic decomposition}

In this section we prove Theorem \ref{thm1}.
First we give a partition of unity related to the auxiliary
function $\rho(g)$. It follows from Lemma \ref{lem2} that $0<
\rho(g)< \infty$ for any $g\in\mathbb H^n$. Therefore,
\begin{eqnarray*}
\mathbb{H}^n= \bigcup_{k=-\infty}^{\infty} \Omega_k
\end{eqnarray*}
where
\begin{eqnarray*} \Omega_k = \{ g \in \mathbb{H}^n:\; 2^{k-1}
< \rho (g) \leq 2^{k} \}.
\end{eqnarray*}
From Lemma \ref{lem4}, we get

\begin{lemma}\label{lem12}
For every $R \geq 2$, if $g \in \Omega_k,\, h \in \Omega_j$ and
$B(g, 2^{k}R) \bigcap B(h, 2^{j}R) \neq \varnothing$, then $|k-j|
\leq C \log_2 R$.
\end{lemma}

\begin{proof}
Choose $g_0 \in B(g, 2^{k}R) \bigcap B(h, 2^{j}R)$ and $g_0 \in
\Omega_{k_0}$. By Lemma \ref{lem4},
\begin{eqnarray*}
&&\frac{\rho(g_0)}{\rho(g)} \leq C \bigg( 1+
\frac{2^{k}R}{\rho(g)} \bigg )^{ \frac{l_0}{l_0+1}} \leq 2^{C
\log_2 R},\\
&&\frac{\rho(g_0)}{\rho(g)} \geq \frac{1}{C} \bigg ( 1+
\frac{2^{k}R}{\rho(g)} \bigg )^{-l_0}\; \geq 2^{-C \log_2 R}.
\end{eqnarray*}
Thus $|k_0 -k| \leq C \log_2 R$. Similarly, $|k_0 -j| \leq C
\log_2 R$. Therefore $|k-j| \leq C \log_2 R$.
\end{proof}

We choose $g_{(k, \alpha)} \in \Omega_k$ such that
\begin{eqnarray*}
\Omega_k \subset \bigcup_{\alpha} B \big( g_{(k, \alpha)},
2^{k-1} \big)
\end{eqnarray*}
and
\begin{eqnarray*}
B \big( g_{(k, \alpha)}, 2^{k-2} \big) \bigcap B \big( g_{(k,
\alpha')}, 2^{k-2} \big) = \varnothing \qquad \text{for}\quad
\alpha \neq \alpha'.
\end{eqnarray*}

From Lemma \ref{lem12}, we get

\begin{lemma}\label{lem13}
For every $(k', \alpha')$ and $R \geq 2$,
\begin{eqnarray*}
\#\, \big\{ (k, \alpha): \; B \big ( g_{(k, \alpha)}, 2^{k}R
\big) \bigcap B \big( g_{(k', \alpha')}, 2^{k'}R \big) \neq
\varnothing \big\} \leq R^C.
\end{eqnarray*}
\end{lemma}
Then we have

\begin{lemma}\label{lem14}
There exists $l_2 >0$ such that, for every $g_{(k', \alpha')}$ and
$l \geq l_2$,
\begin{eqnarray*}
\sum_{(k, \alpha)} \big( 1+ 2^{-k} |\,g^{-1}_{(k', \alpha')}\,
g_{(k, \alpha)}| \big)^{-l} + \sum_{(k, \alpha)} \big( 1+
2^{-k'} |\,g^{-1}_{(k', \alpha')}\, g_{(k, \alpha)}| \big)^{-l}
\leq C.
\end{eqnarray*}
\end{lemma}

Let $B_{(k, \alpha)} = B \big( g_{(k, \alpha)}, 2^{k} \big)$ and
$B^*_{(k, \alpha)} = B \big( g_{(k, \alpha)}, 2^{k+1} \big)$,
and we use these two notations through the article.
We now have the partition of unity.

\begin{lemma}\label{lem15}
There are functions $\xi_{(k, \alpha)} \in C^{\infty}_{c}
(\mathbb{H}^n)$ such that
\begin{itemize}
\item [\rm(a)] $\mathrm{supp}\, \xi_{(k, \alpha)} \subset B_{(k, \alpha)};$
\item [\rm(b)] $0 \leq \xi_{(k, \alpha)}(g) \leq 1\ $ and
               $\ \displaystyle\sum_{(k, \alpha)} \xi_{(k, \alpha)}(g) =1$
               \quad\text{for any} $\, g\in \mathbb{H}^n;$
\item [\rm(c)] $\big| \nabla_{\mathbb{H}^n} \xi_{(k, \alpha)} (g) \big| \leq C\, 2^{-k}$
               \quad\text{for any} $\, g\in \mathbb{H}^n.$
\end{itemize}
\end{lemma}

Now we define the local maximal functions with respect to the
semigroups $\big\{ T^L_s \big\}_{s>0}$ and $\{ T_s \}_{s>0}$ by
\begin{eqnarray*}
\widetilde{M}^L_kf(g) &=& \sup _{0<s \leq 4^{k}} \big | T^L_sf(g) \big|,\\
\widetilde{M}_kf(g) &=& \sup _{0<s \leq 4^{k}} \big| T_sf(g)
\big|.
\end{eqnarray*}
We investigate the relation between $\widetilde{M}^L_kf$ and
$\widetilde{M}_kf$. In practice, we study the local maximal
function $\mathcal{M}_kf$ defined by
\begin{eqnarray*}
\mathcal{M}_kf(g) = \sup _{0<s \leq 4^{k}} \big| T^L_sf(g) -
T_sf(g) \big|.
\end{eqnarray*}

\begin{lemma}\label{lem16}
For every $(k, \alpha)$,
\begin{eqnarray*}
\big\| \mathcal{M}_k (\xi_{(k, \alpha)}f) \big\|_{L^1} \leq C\, \big\|
\xi_{(k, \alpha)}f \big\|_{L^1}.
\end{eqnarray*}
\end{lemma}

\begin{proof}
By Lemma \ref{lem7}, we have
\begin{eqnarray}
\int_{B^*_{(k, \alpha)}} \big|\mathcal{M}_k (\xi_{(k, \alpha)}f)(g) \big|\,dg
&\leq & C \int_{B^*_{(k, \alpha)}} \int_{B_{(k, \alpha)}}
  \frac{ \big| (\xi_{(k, \alpha)}f)(h) \big|}{\rho(g)^{\delta}|h^{-1}g|^{Q- \delta}}\, dh\, dg \nonumber\\
&\leq & C\, \big\| \xi_{(k, \alpha)}f \big\|_{L^1}. \label{a28}
\end{eqnarray}
By $(\ref{a1})$ and $(\ref{a2})$,
\begin{eqnarray*}
\sup _{0<s \leq 4^{k}} \big| K^L_s(g,h)-H_s(g,h) \big|
\leq C\, 2^{-kQ} e^{-A_0 4^{-k} |h^{-1} g|^2}
\qquad \text{if}\ \ |h^{-1}g| > 2^{k}.
\end{eqnarray*}
Therefore,
\begin{eqnarray}
&&\int_{|g^{-1}_{(k, \alpha)}g| \geq 2^{k+1}}
\big| \mathcal{M}_k (\xi_{(k, \alpha)}f)(g) \big|\, dg\nonumber\\
&&\qquad\leq \int_{|g^{-1}_{(k, \alpha)}g| \geq 2^{k+1}}
\int_{B_{(k,\alpha)}} \big| (\xi_{(k, \alpha)}f)(h) \big|\,
\sup_{0<s \leq 4^{k}} \big| K^L_s(g,h)-H_s(g,h) \big|\, dh\, dg\nonumber\\
&&\qquad\leq C\, \big\| \xi_{(k, \alpha)}f \big\|_{L^1}.\label{a29}
\end{eqnarray}
Lemma \ref{lem16} follows from $(\ref{a28})$ and $(\ref{a29})$.
\end{proof}

Define
\begin{eqnarray*}
\mathcal{M}^L_{(k, \alpha)}f(g) = \sup _{0<s \leq 4^{k}} \big |
T^L_s (\xi_{(k, \alpha)}f) (g) - \xi_{(k, \alpha)}(g)\,T^L_sf(g)
\big |.
\end{eqnarray*}

\begin{lemma}\label{lem17}
\begin{eqnarray*}
\sum_{(k, \alpha)} \big\| \mathcal{M}^L_{(k, \alpha)}f
\big\|_{L^1} \leq C\, \| f \|_{L^1}.
\end{eqnarray*}
\end{lemma}

\begin{proof}
Write
\begin{eqnarray*}
\mathcal P^L_{s,\, (k, \alpha),\, (k', \alpha')}f(g)&=& T^L_s (\xi_{(k,
\alpha)}\,
\xi_{(k', \alpha')} f) (g) - \xi_{(k, \alpha)}(g)\,T^L_s (\xi_{(k', \alpha')}f )(g)\\
&=& \int_{B_{(k', \alpha')}} f(h)\, \xi_{(k', \alpha')}(h) \big (
\xi_{(k, \alpha)}(h)- \xi_{(k, \alpha)}(g) \big ) K^L_s(g,h) \, dh
\end{eqnarray*}
and
\begin{eqnarray*}
\mathcal{M}^L_{(k, \alpha),\, (k', \alpha')}f(g) = \sup _{0<s \leq
4^{k}} \big | \mathcal P^L_{s,\, (k, \alpha),\, (k', \alpha')}f(g) \big |.
\end{eqnarray*}
Set
\begin{eqnarray*}
\Theta_{(k, \alpha)}&=& \big\{ (k', \alpha'):\; |\,g^{-1}_{(k',
\alpha')}\, g_{(k, \alpha)}| \leq
A 2^{k} \big\},\\
\Xi_{(k, \alpha)}&=& \big\{ (k', \alpha'):\; |\,g^{-1}_{(k',
\alpha')}\, g_{(k, \alpha)}| > A 2^{k} \big\},
\end{eqnarray*}
where $A>0$ is a fixed constant. By Lemma \ref{lem13}, the number
of elements in $\Theta_{(k, \alpha)}$ is bounded by a constant
independent of $(k, \alpha)$. By Lemma \ref{lem12}, there exists a
constant $A_1 >0$ such that $B_{(k', \alpha')} \subset B(g_{(k,
\alpha)}, A_1 2^{k})\,$ if $\,(k', \alpha') \in \Theta_{(k,
\alpha)}$. Also by Lemma \ref{lem12}, we can take $A$ large enough
such that $B^*_{(k, \alpha)} \bigcap B^*_{(k', \alpha')}=
\varnothing$ when $(k', \alpha') \in \Xi_{(k, \alpha)}$.

Suppose $(k', \alpha') \in \Theta_{(k, \alpha)}$. Using Lemma
\ref{lem15} (c) and the mean value theorem (cf.
\cite[Theorem 1.41]{Folland-Stein}), together with $(\ref{a1})$
and $(\ref{a2})$, we obtain that
\begin{eqnarray*} &&\sup _{0<s
\leq 4^{k}} \big | \big ( \xi_{(k, \alpha)}(h)-
  \xi_{(k, \alpha)}(g) \big ) K^L_s(g,h) \big |\\
&&\qquad\leq C\, 2^{-k} |h^{-1}g| \sup _{0<s\leq 4^{k}} H_s(g,h)\\
&&\qquad\leq \begin{cases} C\, 2^{-k} |h^{-1}g|^{1-Q}, \ \
  &\text{if}\quad |h^{-1}g| \leq 2^{k},\\
   C\, 2^{-k(Q+1)} |h^{-1}g|\, e^{-A_0 4^{-k} |h^{-1} g|^2}, \ \
  &\text{if}\quad |h^{-1}g|> 2^{k}.\end{cases}
\end{eqnarray*}
Therefore,
\begin{eqnarray}
&&\int_{\mathbb{H}^n} \mathcal{M}^L_{(k, \alpha),\, (k', \alpha')}f(g) \, dg\nonumber\\
&&\qquad\leq \int_{\mathbb{H}^n} \int_{B_{(k', \alpha')}} \big|
f(h)\,\xi_{(k', \alpha')}(h) \big|\, \sup _{0<s \leq 4^{k}}
\big|\big (\xi_{(k, \alpha)}(h)-\xi_{(k, \alpha)}(g) \big) K^L_s(g,h) \big|\, dh\, dg\nonumber\\
&&\qquad\leq C\, \big\| f \big\|_{L^1 (B(g_{(k, \alpha)}, A_1
2^{k}))}.\label{a30}
\end{eqnarray}

Let $(k', \alpha') \in \Xi_{(k, \alpha)}$. It is easy to see that
$\xi_{(k', \alpha')}(h) \xi_{(k, \alpha)}(h)=0$, $\rho(g) \sim
2^k$ and $|h^{-1}g| \sim |\,g^{-1}_{(k', \alpha')}\, g_{(k,
\alpha)}|\,$ for $g \in B_{(k, \alpha)}$ and $h \in B_{(k',
\alpha')}$. By Lemma \ref{lem7}, we get
\begin{eqnarray}
&&\int_{\mathbb{H}^n} \mathcal{M}^L_{(k, \alpha),\, (k', \alpha')}f(g) \, dg\nonumber\\
&&\qquad\leq \int_{B_{(k, \alpha)}} \int_{B_{(k', \alpha')}} \big|
f(h) \big| \frac{C_l\, \xi_{(k', \alpha')}(h)\, \xi_{(k,
\alpha)}(g)} {\big ( 1+|h^{-1}g|\, \rho(g)^{-1} \big )^l |h^{-1}g|^{Q}}\, dh\, dg\nonumber\\
&&\qquad \leq \frac{C} {\big ( 1+2^{-k} |\,g^{-1}_{(k',
\alpha')}\, g_{(k, \alpha)}|\big )^{l}} \int_{B_{(k',\alpha')}}
\big| f(h) \big|\, \xi_{(k', \alpha')}(h) \bigg( \int_{B_{(k, \alpha)}}
\frac {dg}{|h^{-1}g|^Q}\bigg)dh\nonumber\\
&&\qquad\leq C \big ( 1+2^{-k} |\,g^{-1}_{(k', \alpha')}\,
g_{(k,\alpha)}|\big)^{-l}\, \big\| \xi_{(k', \alpha')}f
\big\|_{L^1}.\label{a31}
\end{eqnarray}
We have
\begin{eqnarray}
\sum_{(k, \alpha)} \big\| \mathcal{M}^L_{(k, \alpha)}f
\big\|_{L^1} &\leq& \sum_{(k, \alpha)} \sum_{(k', \alpha')} \big\|
\mathcal{M}^L_{(k, \alpha),\, (k', \alpha')}f \big\|_{L^1}\nonumber\\
&=& \sum_{(k, \alpha)} \sum_{(k', \alpha') \in \Theta_{(k,
\alpha)}} + \sum_{(k, \alpha)} \sum_{(k', \alpha') \in \Xi_{(k,
\alpha)}}= J_1 + J_2.\label{a32}
\end{eqnarray}
By $(\ref{a30})$ and Lemma \ref{lem13}, we get
\begin{eqnarray}\label{a33}
J_1 \leq C \sum_{(k, \alpha)} \big\| f \big\|_{L^1 (B(g_{(k,
\alpha)}, A_1 2^{k}))} \leq C\, \big\| f \big\|_{L^1}.
\end{eqnarray}
Taking $l \geq l_2$, by $(\ref{a31})$ and Lemma \ref{lem14}, we get
\begin{equation} \label{a34}
J_2 \leq C \sum_{(k, \alpha)} \sum_{(k', \alpha') \in \Xi_{(k,
\alpha)}} \big ( 1+2^{-k} |\,g^{-1}_{(k', \alpha')}\, g_{(k,
\alpha)}|\big )^{-l}\, \big\| \xi_{(k', \alpha')}f \big\|_{L^1}
\leq C\, \big\| f \big\|_{L^1}. 
\end{equation}
Combination of $(\ref{a32})$, $(\ref{a33})$ and $(\ref{a34})$
gives Lemma \ref{lem17}.
\end{proof}

We are ready to prove Theorem \ref{thm1}.

{\flushleft \it Proof of Theorem \ref{thm1}.} Given $f \in
H^1_L(\mathbb{H}^n)$, since
\begin{eqnarray*}
\big\| \widetilde{M}_k (\xi_{(k, \alpha)}f) \big\|_{L^1} &\leq&
\big\| \widetilde{M}^L_k (\xi_{(k, \alpha)}f) \big\|_{L^1} +
\big\| \mathcal{M}_k (\xi_{(k, \alpha)}f) \big\|_{L^1}\\
&\leq& \big\| (\xi_{(k, \alpha)}) \widetilde{M}^L_k f \big\|_{L^1}
+ \big\| \mathcal{M}^L_{(k, \alpha)}f \big\|_{L^1} + \big\|
\mathcal{M}_k (\xi_{(k, \alpha)}f) \big\|_{L^1},
\end{eqnarray*}
by Lemma \ref{lem16} and Lemma \ref{lem17},
\begin{eqnarray}\label{a35}
\sum_{(k, \alpha)} \big\| \widetilde{M}_k (\xi_{(k, \alpha)}f)
\big\|_{L^1} \leq C \big ( \big\| M^L f \big\|_{L^1} + \big\| f
\big\|_{L^1} \big ) \leq C\, \big\| f \big\|_{H^1_L}.
\end{eqnarray}
As pointed out in Remark \ref{rem3}, this means every $\xi_{(k,
\alpha)}f \in h^1_k(\mathbb{H}^n)$ and
\begin{eqnarray}\label{a36}
\xi_{(k, \alpha)}f = \sum_{i} \lambda^{(k, \alpha)}_{i}\, a^{(k,
\alpha)}_{i}\qquad\text{in}\ L^1(\mathbb{H}^n),
\end{eqnarray}
where $a^{(k, \alpha)}_{i}$ are $h^{1,q}_k$-atoms and
\begin{eqnarray}\label{a37}
\sum_{i} \big | \lambda^{(k, \alpha)}_{i} \big | \leq C\, \big\|
\widetilde{M}_k (\xi_{(k, \alpha)}f) \big\|_{L^1}.
\end{eqnarray}
We also note that each $h^{1,q}_k$-atom $a^{(k, \alpha)}_{i}$ 
without the cancellation condition is supported on the ball
$B_{(k, \alpha)}^* = B\big( g_{(k, \alpha)}, 2^{k+1} \big)$.
Therefore all $h^{1,q}_k$-atoms $a^{(k,
\alpha)}_{i}$ are $H^{1,q}_L$-atoms. By $(\ref{a36})$,
$(\ref{a37})$ and $(\ref{a35})$, we get
\begin{eqnarray*}
f = \sum_{(k, \alpha)} \sum_{i} \lambda^{(k, \alpha)}_{i}\, a^{(k,
\alpha)}_{i}
\end{eqnarray*}
and
\begin{eqnarray*}
\sum_{(k, \alpha)} \sum_{i} \big | \lambda^{(k, \alpha)}_{i} \big
| \leq C \sum_{(k, \alpha)} \big\| \widetilde{M}_k (\xi_{(k,
\alpha)}f) \big\|_{L^1} \leq C\, \big\| f \big\|_{H^1_L}.
\end{eqnarray*}

To prove the converse, we only have to show that, for every
$H^{1,q}_L$-atom $a\,$,
\begin{eqnarray}\label{a38}
\big\| M^L a \big\|_{L^1} \leq C.
\end{eqnarray}
Let $ a $ be an $H^{1,q}_L$-atom supported on a ball $B(g_0, r)$.
If $r > \rho(g_0)$, we further decompose $a$ as follows. Set
\begin{eqnarray*}
\Theta_{a} = \big\{ (k, \alpha):\; B \big( g_{(k, \alpha)},
\rho(g_{(k, \alpha)}) \big) \bigcap B(g_0, r) \neq \varnothing
\big\}.
\end{eqnarray*}
Then we have
\begin{eqnarray}\label{a39}
\sum_{(k, \alpha) \in \Theta_{a}} \big| B \big( g_{(k, \alpha)},
\rho(g_{(k, \alpha)}) \big) \big| \leq C\, \big| B(g_0, r) \big|.
\end{eqnarray}
In fact, if $\rho(g_{(k, \alpha)}) > r$ and $(k, \alpha) \in
\Theta_{a}$, then $g_{0} \in B \big( g_{(k, \alpha)}, 2\rho(g_{(k,
\alpha)}) \big)$. By Lemma \ref{lem4},
\begin{eqnarray*}
\rho(g_{(k, \alpha)}) \leq C\, \rho(g_0) \leq C\, r.
\end{eqnarray*}
Therefore, $B \big( g_{(k, \alpha)}, \rho(g_{(k, \alpha)}) \big)
\subset B(g_0, C\,r)$, and $(\ref{a39})$ follows from the finite
overlaps property of $B \big( g_{(k, \alpha)}, \rho(g_{(k,
\alpha)}) \big)$. Let $\chi_{(k, \alpha)}$ be the characteristic
function of $B \big( g_{(k, \alpha)}, \rho(g_{(k, \alpha)})
\big)$. Set
\begin{eqnarray*}
\zeta_{(k, \alpha)}(g)= \frac{\chi_{(k, \alpha)} (g)}{\sum\limits_{(k',
\alpha') \in \Theta_{a}} \chi_{(k', \alpha')} (g)}
\end{eqnarray*}
and write
\begin{eqnarray*}
\zeta_{(k, \alpha)}(g)\,a(g)=  \lambda_{(k, \alpha)}\, a_{(k,
\alpha)}(g),
\end{eqnarray*}
where
\begin{eqnarray*}
\lambda_{(k, \alpha)}= \big| B \big( g_{(k, \alpha)}, \rho(g_{(k,
\alpha)}) \big) \big|^{\frac{1}{q'}} \, \big\| \zeta_{(k, \alpha)}
a \big\|_{L^{q}}.
\end{eqnarray*}
Then $a_{(k, \alpha)}$ is an $H^{1,q}_L$-atom supported on $B
\big( g_{(k, \alpha)}, \rho(g_{(k, \alpha)}) \big)$ and
\begin{eqnarray*}
a(g)= \sum_{(k, \alpha) \in \Theta_{a}} \lambda_{(k, \alpha)}\,
a_{(k, \alpha)}(g).
\end{eqnarray*}
By the H\"{o}lder inequality and $(\ref{a39})$,
\begin{eqnarray*}
\sum_{(k, \alpha) \in \Theta_{a}} |\lambda_{(k, \alpha)}| &\leq&
\bigg( \sum_{(k, \alpha) \in \Theta_{a}} \big| B \big( g_{(k,
\alpha)}, \rho(g_{(k, \alpha)}) \big) \big| \bigg)^{\frac{1}{q'}}
\bigg( \sum_{(k, \alpha) \in \Theta_{a}} \big\| \zeta_{(k,
\alpha)} a \big\|^q_{L^{q}} \bigg)^{\frac{1}{q}}\\
&\leq& C\, \big| B(g_0, r) \big|^{\frac{1}{q'}}\, \big\| a
\big\|_{L^{q}} \leq C.
\end{eqnarray*}
Thus, without loss of generality, we may assume that $a$ is an
$H^{1,q}_L$-atom supported on a ball $B(g_0, r)$ satisfying $r
\leq \rho(g_0)$. Let $g_0 \in \Omega_k$. By the same argument of
Lemma \ref{lem16}, we have $\big\| \mathcal{M}_k a \big\|_{L^1}
\leq C \big\| a \big\|_{L^1} \leq C$. It is clear that $\big\|
\widetilde{M}_k\, a \big\|_{L^1} \leq C$. Hence
\begin{eqnarray}\label{a40}
\big\|\widetilde{M}^L_k a \big\|_{L^1} \leq C.
\end{eqnarray}
By $(\ref{a1})$ and $(\ref{a2})$,
\begin{eqnarray*}
\sup_{4^{k} < s < \infty} K^L_s(g,h) \leq C\, 2^{-kQ},
\end{eqnarray*}
which yields
\begin{eqnarray*}
\sup_{4^{k} < s < \infty} \big | T^L_s a(g) \big |  \leq C\,
2^{-kQ}
\end{eqnarray*}
and hence
\begin{eqnarray}\label{a41}
\int_{B(g_0, 2^{k+1})} \sup_{4^{k} < s < \infty} \big | T^L_s a(g)
\big |\, dg \leq C.
\end{eqnarray}
Note that $r \leq 2^k$ and $\rho(h) \sim 2^k$ whenever $h \in
B(g_0, r)$. By Lemma \ref{lem7}, we have
\begin{eqnarray}
&&\int_{|g^{-1}_0 g| \geq 2^{k+1}} \sup_{4^{k} < s < \infty}
\big| T^L_s a(g) \big|\, dg\nonumber\\
&&\qquad \leq C \int_{|g^{-1}_0 g| \geq 2^{k+1}} \int_{B(g_0, r)}
\frac{|a(h)|}{\big ( 1+|h^{-1}g|\, \rho(h)^{-1} \big )
|h^{-1}g|^{Q}}\, dh\, dg \leq C.\label{a42}
\end{eqnarray}
We obtain $(\ref{a38})$ from $(\ref{a40})$, $(\ref{a41})$ and
$(\ref{a42})$. This completes the proof of Theorem \ref{thm1}. \qed

\begin{remark}\label{rem4}
{\rm As we showed above, we may add a restriction in the definition of
$H^{1,q}_L$-atom as follows. If $\,a$ is an $H^{1,q}_L$-atom
supported on a ball $B(g_0, r)$, then $r \leq \rho(g_0)$.}
\end{remark}

\begin{remark}\label{rem5}
{\rm We have seen that $H^1_L(\mathbb{H}^n)$ is closely related to the
local Hardy space $h^1(\mathbb{H}^n)$. From Theorem \ref{thm1},
Theorem \ref{thm4} and Remark \ref{rem1}, it is easy to get the
following inclusion relations. If $V \leq C$ then
$H^1_L(\mathbb{H}^n) \subset h^1(\mathbb{H}^n)$. If $V \geq
\frac{1}{C}$ then  $h^1(\mathbb{H}^n) \subset
H^1_L(\mathbb{H}^n)$. Specifically, $H^1_L(\mathbb{H}^n) =
h^1(\mathbb{H}^n)$ when $V \sim 1$.}
\end{remark}


\section{Riesz transforms characterization }

In this section we deal with the Riesz transforms $R^L_j$ 
associated with the Schr\"odinger operator $L$.
The atomic decomposition is useful to establish the boundedness of
operators on Hardy spaces. But in general, as showed in
\cite{Bownik} (also see \cite{Meyer}), it is not enough to conclude 
that an operator extends to a bounded operator on the whole Hardy space 
by verifying its uniform boundedness on all atoms. The key point is that two quasi-norms
corresponding to finite and infinite atomic decompositions are not
equivalent. We give the following lemma.

\begin{lemma}\label{lem18}
Suppose that $\,T$ is a bounded sublinear operator from $L^1(\mathbb{H}^n)$ to
$L^{1,\infty}(\mathbb{H}^n)$ such that $\big\| Ta \big\|_{L^1} \leq C$
for any $H^{1, q}_L$-atom $a$. Then $\,T$ is bounded from
$H^1_L(\mathbb{H}^n)$ to $L^1(\mathbb{H}^n)$.
\end{lemma}

\begin{proof}
Suppose $f \in H^1_L(\mathbb{H}^n)$. By Theorem
\ref{thm1}, we write
\begin{eqnarray*}
f= \sum^{\infty}_{j=1} \lambda_j\, a_j \quad \text{with} \quad
\sum^{\infty}_{j=1} |\lambda_j|< C\, \big\| f \big\|_{H^1_L},
\end{eqnarray*}
where $a_j$ are $H^{1,q}_L$-atoms. Set
\begin{eqnarray*}
f_m= \sum^{m}_{j=1} \lambda_j\, a_j.
\end{eqnarray*}
It is clear that $f_m$ converges to $f$ in $L^1(\mathbb{H}^n)$.
The sublinearity of $T$ yields $\big| |Tf_m|-|Tf| \big| \le |T(f_m-f)|$, and
hence $|Tf_m|$ converges to $|Tf|$ in $L^{1,\infty}(\mathbb{H}^n)$. There
exists a subsequence $\{ |Tf_{m_k}| \}$ such that
$\lim\limits_{k \to \infty} |Tf_{m_k}(g)|=|Tf(g)|$ almost everywhere
$g \in \mathbb{H}^n$. Since
\begin{eqnarray*}
\big\| Tf_{m_k} \big\|_{L^1} \leq \sum^{m_k}_{j=1} |\lambda_j|\,
\big\| Ta_j \big\|_{L^1} \leq C\, \big\| f \big\|_{H^1_L},
\end{eqnarray*}
by Fatou's lemma we get
\begin{eqnarray*}
\big\| Tf \big\|_{L^1} \leq C\, \big\| f \big\|_{H^1_L}.
\end{eqnarray*}
This proves Lemma \ref{lem18}.
\end{proof}

We are going to prove Theorem \ref{thm2}.

{\flushleft \it Proof of Theorem \ref{thm2}.}
By the Calder\'on--Zygmund decomposition (cf.
\cite{Folland-Stein, Stein}), given $f \in
L^1(\mathbb{H}^n)\,$ and $\, \lambda > 0$, we have the
decomposition $f= f_1+f_2$ with $f_2= \sum_k b_k$, such that
\begin{itemize}
\item [\rm(i)] $\big| f_1(g) \big| \leq C\, \lambda\ $ for almost everywhere
               $\ g \in \mathbb{H}^n$;
\item [\rm(ii)] each $b_k$ is supported in a ball $B_k$,
                $\displaystyle\int_{B_k} \big| b_k(g) \big|\, dg \leq C\, \lambda |B_k|$,
                and $\displaystyle\int_{B_k} b_k(g)\, dg=0$;
\item [\rm(iii)] $\{B_k\}$ is a finitely overlapping family and
                 $\displaystyle\sum_k |B_k| \leq \frac{C}{\lambda}\, \big\| f \big\|_{L^1}$.
\end{itemize}
Since $R^L_j$ is bounded on $L^2(\mathbb{H}^n)$, it is clear that
\begin{eqnarray}\label{a43}
\big | \big\{ g \in \mathbb{H}^n :\; |R^L_jf_1(g)| >
\frac{\lambda}{2}\, \big\} \big | \leq \frac{C}{\lambda^2}\,
\big\| f_1 \big\|^2_{L^2}\, \leq \frac{C}{\lambda}\, \big\| f
\big\|_{L^1}.
\end{eqnarray}
Let $B_k=B(g_k, r_k)$. Set $B^*_k=B(g_k, 2r_k)$ and $\Omega =
\bigcup_k B^*_k$. Then
\begin{eqnarray}\label{a44}
| \Omega | \leq C \sum_k |B_k| \leq \frac{C}{\lambda}\, \big\| f
\big\|_{L^1}.
\end{eqnarray}
We only need to consider $R^L_jf_2(g)$ for $g \in \Omega^c$. If
$r_k \geq \rho(g_k)$, then $\rho(h) \leq C\,r_k$ for any $h \in
B_k$. By Lemma \ref{lem8}, we get
\begin{eqnarray*}
\int_{|g^{-1}_k g| \geq 2r_k} \big| R^L_j b_k(g) \big|\, dg \leq
\int_{|g^{-1}_k g| \geq 2r_k} \int_{B_k} \big| R^L_j(g,h) \big|\,
\big| b_k(h) \big|\, dh\,dg \leq C\, \big\| b_k \big\|_{L^1}.
\end{eqnarray*}
If $r_k < \rho(g_k)$, then $\rho(h) \sim \rho(g_k)$ for any $h \in
B_k$. Since $R_j(g)$ is a Calder\'on--Zygmund kernel, by
Lemmas \ref{lem8} and \ref{lem9} we obtain
\begin{eqnarray*}
\int_{|g^{-1}_k g| \geq 2r_k} \big| R^L_j b_k(g) \big|\, dg &\leq&
\int_{2r_k \leq |g^{-1}_k g|< 2\rho(g_k)} \big| R^L_j b_k(g)
\big|\, dg + \int_{|g^{-1}_k g| \geq 2\rho(g_k)} \big| R^L_j b_k(g) \big|\, dg\\
&\leq& \int_{2r_k \leq |g^{-1}_k g|< 2\rho(g_k)} \int_{B_k}
\big| R^L_j(g,h)-R_j(g,h) \big|\, \big| b_k(h) \big|\, dh\,dg\\
&& + \int_{2r_k \leq |g^{-1}_k g|< 2\rho(g_k)} \int_{B_k}
\big| R_j(g,h)-R_j(g,g_k) \big|\, \big| b_k(h) \big|\, dh\,dg\\
&& + \int_{|g^{-1}_k g| \geq 2 \rho(g_k)} \int_{B_k}
\big| R^L_j(g,h)\big|\, \big| b_k(h) \big|\, dh\,dg\\
&\leq& C\, \big\| b_k \big\|_{L^1}.
\end{eqnarray*}
In any case we have
\begin{eqnarray}\label{a45}
\big\| R^L_j b_k \big\|_{L^1((B^*_k)^c)} \leq C\, \big\| b_k
\big\|_{L^1}.
\end{eqnarray}
Then
\begin{eqnarray*}
\int_{\Omega^c} \big| R^L_jf_2(g) \big|\,dg \leq \sum_k \big\|
R^L_j b_k \big\|_{L^1((B^*_k)^c)} \leq C \sum_k \big\| b_k
\big\|_{L^1} \leq C \lambda \sum_k |B_k| \leq C\, \big\| f
\big\|_{L^1}.
\end{eqnarray*}
Therefore
\begin{eqnarray}\label{a46}
\big | \big\{ g \in \Omega^c :\; |R^L_jf_2(g)| >
\frac{\lambda}{2}\, \big\} \big | \leq \frac{C}{\lambda}\, \big\|
f \big\|_{L^1}.
\end{eqnarray}
Theorem \ref{thm2} follows from the combination of $(\ref{a43})$,
$(\ref{a44})$ and $(\ref{a46})$. \qed

The proof of Theorem \ref{thm3} is similar to the one in Theorem \ref{thm1}.
We will keep the notations used in former sections.
As consequences of Lemmas \ref{lem8} and \ref{lem9}, we have

\begin{lemma}\label{lem19}
For every $(k, \alpha)$,
\begin{eqnarray*}
\big\| R^L_j(\xi_{(k, \alpha)}f) \big\|_{L^1((B^*_{(k, \alpha)})^c)} \leq
C\, \big\| \xi_{(k, \alpha)}f \big\|_{L^1}.
\end{eqnarray*}
\end{lemma}

\begin{lemma}\label{lem20}
For every $(k, \alpha)$,
\begin{eqnarray*}
\big\| R^L_j(\xi_{(k, \alpha)}f) - R_j(\xi_{(k, \alpha)}f)
\big\|_{L^1(B^*_{(k, \alpha)})} \leq C\, \big\| \xi_{(k, \alpha)}f \big\|_{L^1}.
\end{eqnarray*}
\end{lemma}

Similar to Lemma \ref{lem17}, we have

\begin{lemma}\label{lem21}
\begin{eqnarray*}
\sum_{(k, \alpha)} \big\| R^L_j (\xi_{(k, \alpha)}f) - \xi_{(k,
\alpha)}\,R^L_j f \big\|_{L^1} \leq C\, \big\| f \big\|_{L^1}.
\end{eqnarray*}
\end{lemma}

\begin{proof}
Given $(k, \alpha)$, we write
\begin{eqnarray*}
\mathcal Q^L_{j,\, (k, \alpha),\, (k', \alpha')}f(g)&=& R^L_j (\xi_{(k,
\alpha)}\, \xi_{(k', \alpha')} f) (g) - \xi_{(k, \alpha)}(g)\,
R^L_j (\xi_{(k', \alpha')}f)(g)\\
&=& \int_{B_{(k', \alpha')}} f(h)\, \xi_{(k', \alpha')}(h) \big(
\xi_{(k, \alpha)}(h)- \xi_{(k, \alpha)}(g) \big) R^L_j(g,h)\, dh.
\end{eqnarray*}
Suppose $(k', \alpha') \in \Theta_{(k, \alpha)}$ and $h \in
B_{(k', \alpha')}$. Then $\rho(h) \sim 2^{k'} \sim 2^{k}$.
By Lemma \ref{lem8},
\begin{eqnarray*}
&& \int_{\mathbb{H}^n} \big| \big( \xi_{(k, \alpha)}(h)-
\xi_{(k, \alpha)}(g) \big) R^L_j(g,h) \big| \, dg\\
&&\qquad \leq 2 \int_{|h^{-1}g| \geq 2^{k'}} \big| R^L_j(g,h) \big|\, dg + C
\int_{B(h, 2^{k'})} 2^{-k'} |h^{-1}g|\, \big| R^L_j(g,h) \big|\, dg\\
&&\qquad\leq C + C \int_{B(h, 2^{k'})} 2^{-k'} |h^{-1} g|\, \big| R^L_j(g,h) \big|\, dg.
\end{eqnarray*}
Using $(\ref{a12})$ with $l=1$, we get
\begin{eqnarray*}
&&\int_{B(h,\, 2^{k'})} 2^{-k'} |h^{-1} g|\, \big| R^L_j(g,h) \big|\, dg\\
&&\qquad \leq C \int_{B(h, 2^{k'})} \frac{2^{-k'}\, dg}{|h^{-1}g|^{Q-1}}
+ C \int_{B(h, 2^{k'})} \bigg( \frac{1}{|h^{-1}g|^{Q-1}}
\int_{B(g, \frac{|h^{-1}g|}{2})} \frac{V(w)\, dw}{|g^{-1}w|^{Q-1}}\bigg)\, dg \\
&&\qquad \leq C + C \int_{B(h, 2^{k'})} \bigg(
\frac{1}{|h^{-1}g|^{Q-1}} \int_{B(g, \frac{|h^{-1}g|}{2})}
\frac{V(w)\, dw}{|g^{-1}w|^{Q-1}} \bigg)\, dg.
\end{eqnarray*}
Note that $p'_{_0}(Q-1) < Q\,$ as $\,q_{_0} >\frac{Q}{2}$. By the H\"older inequality
and the boundedness of fractional integrals, we obtain
\begin{eqnarray}
&&\int_{B(h, 2^{k'})} \bigg( \frac{1}{|h^{-1}g|^{Q-1}}
\int_{B(g, \frac{|h^{-1}g|}{2})} \frac{V(w)\, dw}
 {|g^{-1}w|^{Q-1}} \bigg) \, dg\nonumber\\
&&\qquad \leq \bigg( \int_{B(h, 2^{k'})}
\frac{dg}{|h^{-1}g|^{p'_{_0}(Q-1)}} \bigg )^{\frac{1}{p'_{_0}}} \left(
\int_{B(h, 2^{k'})} \bigg( \int_{B(g, \frac{|h^{-1}g|}{2})}
\frac{V(w)\, dw}{|g^{-1}w|^{Q-1}} \bigg)^{p_{_0}}\, dg  \right)^{\frac{1}{p_{_0}}}\nonumber\\
&&\qquad \leq C 2^{-k'(\frac{Q}{p_{_0}}-1)} \bigg( \int_{B(h,
2^{k'+1})}
V(g)^{q_{_0}}\, dg \bigg)^{\frac{1}{q_{_0}}}\nonumber\\
&&\qquad \leq C 2^{-k'(Q-2)} \int_{B(h, 2^{k'+1})} V(g)\, dg\nonumber\\
&&\qquad \leq C,\label{a47}
\end{eqnarray}
where we used the $B_{q_{_0}}$ condition and Lemma \ref{lem3} in
the last two inequalities. The above three estimates yield
\begin{eqnarray*}
\int_{\mathbb{H}^n} \big| \big ( \xi_{(k, \alpha)}(h)- \xi_{(k,
\alpha)}(g) \big ) R^L_j(g,h) \big| \, dg \leq C.
\end{eqnarray*}
Note that $B_{(k', \alpha')} \subset B(g_{(k, \alpha)}, A_1
2^{k})\,$ when $(k', \alpha') \in \Theta_{(k, \alpha)}$. Hence,
\begin{eqnarray}\label{a48}
\big\| \mathcal Q^L_{j,\, (k, \alpha),\, (k', \alpha')}f \big\|_{L^1} \leq C\, \big\|
\xi_{(k', \alpha')}f \big\|_{L^1} \leq C\, \big\| f \big\|_{L^1 (B(g_{(k,
\alpha)}, A_1 2^{k}))}.
\end{eqnarray}
Let $(k', \alpha') \in \Xi_{(k, \alpha)}$ and $h \in B_{(k',
\alpha')}$. We have $\xi_{(k, \alpha)}(h)=0$, $\rho(h) \sim
2^{k'}$ and $|h^{-1}g| \sim |g^{-1}_{(k', \alpha')}\, g_{(k,
\alpha)}|\,$ for $g \in B_{(k, \alpha)}$. By $(\ref{a12})$, we get
\begin{eqnarray*}
&&\int_{\mathbb{H}^n} \big| \big( \xi_{(k, \alpha)}(h)-
\xi_{(k, \alpha)}(g) \big) R^L_j(g,h) \big| \, dg\\
&&\qquad \leq \int_{B_{(k, \alpha)}} \frac{C_l}{\big( 1+ |h^{-1}g|\,
\rho(h)^{-1} \big)^l} \bigg( \frac{1}{|h^{-1}g|^{Q}} +
\frac{1}{|h^{-1}g|^{Q-1}} \int_{B(g, \frac{|h^{-1}g|}{2})}
\frac{V(w)\, dw}{|g^{-1}w|^{Q-1}} \bigg) \, dg\\
&&\qquad \leq \frac{C}{\big( 1+ 2^{-k'}|g^{-1}_{(k', \alpha')}\, g_{(k,\alpha)}|
\big)^l} \bigg( 1+ \frac{1}{|g^{-1}_{(k', \alpha')}\, g_{(k,
\alpha)}|^{Q-1}} \int_{B_{(k, \alpha)}} \int_{B(g,
\frac{|h^{-1}g|}{2})} \frac{V(w)\, dw}{|g^{-1}w|^{Q-1}}\, dg \bigg).
\end{eqnarray*}
Similar to $(\ref{a47})$, by the H\"older inequality,
the boundedness of fractional integrals and the $B_{q_{_0}}$ condition,
we obtain
\begin{eqnarray*}
&&\frac{1}{|g^{-1}_{(k', \alpha')}\, g_{(k, \alpha)}|^{Q-1}}
\int_{B_{(k, \alpha)}} \int_{B(g,
\frac{|h^{-1}g|}{2})} \frac{V(w)\, dw}{|g^{-1}w|^{Q-1}}\, dg\\
&&\qquad \leq C\, |g^{-1}_{(k', \alpha')}\, g_{(k, \alpha)}|^{2- Q}
\int_{B(g_{(k', \alpha')},\, \frac 94 |g^{-1}_{(k', \alpha')}\, g_{(k,\alpha)}|)} V(g)\, dg\\
&&\qquad \leq C \big( 1+ 2^{-k'}|g^{-1}_{(k', \alpha')}\, g_{(k,\alpha)}| \big)^{l_1},
\end{eqnarray*}
where we used Lemma \ref{lem5} for the last inequality.
Thus,
\begin{eqnarray*}
\int_{\mathbb{H}^n} \big|  \big( \xi_{(k, \alpha)}(h)- \xi_{(k,
\alpha)}(g) \big) R^L_j(g,h) \big|\, dg \leq C \big( 1+
2^{-k'}|g^{-1}_{(k', \alpha')}\, g_{(k, \alpha)}| \big)^{-l+l_1},
\end{eqnarray*}
and hence
\begin{equation}\label{a49}
\big\| \mathcal Q^L_{j,\, (k, \alpha),\, (k', \alpha')}f \big\|_{L^1}
\leq C \big( 1+ 2^{-k'}|g^{-1}_{(k', \alpha')}\, g_{(k, \alpha)}| \big)^{-l
+l_1} \big\| \xi_{(k', \alpha')}f \big\|_{L^1}.
\end{equation}

We have
\begin{eqnarray}
\sum_{(k, \alpha)} \| R^L_j (\xi_{(k, \alpha)}f) - \xi_{(k,
\alpha)}\,R^L_j f \|_{L^1} &\leq& \sum_{(k, \alpha)} \sum_{(k',
\alpha')} \| \mathcal Q^L_{j,\, (k, \alpha),\, (k', \alpha')}f
\|_{L^1}\nonumber\\
&=& \sum_{(k, \alpha)} \sum_{(k', \alpha') \in \Theta_{(k,
\alpha)}} + \sum_{(k, \alpha)} \sum_{(k', \alpha') \in \Xi_{(k,
\alpha)}}\nonumber\\
&=& J_1 + J_2.\label{a50}
\end{eqnarray}
By $(\ref{a48})$ and Lemma \ref{lem13}, we get
\begin{eqnarray}\label{a51}
J_1 \leq C \sum_{(k, \alpha)} \big\| f\big \|_{L^1 (B(g_{(k, \alpha)}, A_1
2^{-k}))} \leq C\, \big\| f \big\|_{L^1}.
\end{eqnarray}
Taking $l \geq l_1 + l_2$, by $(\ref{a49})$ and Lemma \ref{lem14},
we get
\begin{equation}\label{a52}
J_2 \leq C \sum_{(k, \alpha)} \sum_{(k', \alpha') \in \Xi_{(k,
\alpha)}} \big( 1+2^{-k'} |\,g^{-1}_{(k', \alpha')}\, g_{(k,
\alpha)}| \big)^{-l +l_0}\, \big\| \xi_{(k', \alpha')}f \big\|_{L^1} \leq
C\, \big\| f \big\|_{L^1}.
\end{equation}
Lemma \ref{lem21} follows from the combination of $(\ref{a50})$,
$(\ref{a51})$ and $(\ref{a52})$.
\end{proof}

Now we give the proof of Theorem \ref{thm3}.

{\flushleft \it Proof of Theorem \ref{thm3}.} Assume first that
\begin{eqnarray*}
\big\| f \big\|_{L^1} + \sum_{j=1}^{2n} \big\| R^L_jf \big\|_{L^1} < \infty.
\end{eqnarray*}
Since
\begin{eqnarray*}
&&\big\| R_j(\xi_{(k, \alpha)}f) \big\|_{L^1(B^*_{(k, \alpha)})}\\
&&\qquad \leq \big\| R^L_j(\xi_{(k, \alpha)}f) - R_j(\xi_{(k, \alpha)}f)
\big\|_{L^1(B^*_{(k, \alpha)})} + \big\| R^L_j(\xi_{(k,
\alpha)}f) \big\|_{L^1(B^*_{(k, \alpha)})}\\
&&\qquad \leq \big\| R^L_j(\xi_{(k, \alpha)}f) - R_j(\xi_{(k, \alpha)}f)
\big\|_{L^1(B^*_{(k, \alpha)})} + \big\| R^L_j(\xi_{(k, \alpha)}f)
\big\|_{L^1((B^*_{(k, \alpha)})^c)}\\
&&\qquad \quad + \big\| R^L_j (\xi_{(k, \alpha)}f) - \xi_{(k,
\alpha)}\,R^L_j f \big\|_{L^1} + \big\| \xi_{(k, \alpha)}\,R^L_j f \big\|_{L^1},
\end{eqnarray*}
Lemmas $\ref{lem19}-\ref{lem21}$ give
\begin{eqnarray*}
\sum_{(k, \alpha)} \big\| R_j(\xi_{(k, \alpha)}f) \big\|_{L^1(B^*_{(k,
\alpha)})} \leq C \big ( \big\| f \big\|_{L^1} +  \big\| R^L_jf \big\|_{L^1} \big ).
\end{eqnarray*}
Note that
\begin{eqnarray*}
\big\| R_j(\xi_{(k, \alpha)}f) \big\|_{L^1(B^*_{(k, \alpha)})} &=& \big\|
\widetilde{R}^{[k+3]}_j(\xi_{(k, \alpha)}f) \big\|_{L^1(B^*_{(k, \alpha)})},\\
\big\| \widetilde{R}^{[k]}_j(\xi_{(k, \alpha)}f) \big\|_{L^1} &=& \big\|
\widetilde{R}^{[k]}_j(\xi_{(k, \alpha)}f) \big\|_{L^1(B^*_{(k, \alpha)})}.
\end{eqnarray*}
Since $\big\| R^{[k]}_j(\cdot) - R^{[k+3]}_j(\cdot) \big\|_{L^1} \leq C$, we have
\begin{eqnarray*}
\big\| \widetilde{R}^{[k]}_j(\xi_{(k, \alpha)}f) -
\widetilde{R}^{[k+3]}_j(\xi_{(k, \alpha)}f) \big\|_{L^1(B^*_{(k,
\alpha)})} \leq C\, \big\| \xi_{(k, \alpha)}f \big\|_{L^1},
\end{eqnarray*}
which yields
\begin{eqnarray}
&&\sum_{(k, \alpha)} \big\| \widetilde{R}^{[k]}_j(\xi_{(k,\alpha)}f) \big\|_{L^1} \nonumber\\
&&\qquad \leq \sum_{(k, \alpha)} \big\| \widetilde{R}^{[k]}_j(\xi_{(k,\alpha)}f)
- \widetilde{R}^{[k+3]}_j(\xi_{(k, \alpha)}f)
\big\|_{L^1(B^*_{(k, \alpha)})} + \sum_{(k, \alpha)} \big\| R_j(\xi_{(k,\alpha)}f)
\big\|_{L^1(B^*_{(k, \alpha)})} \nonumber\\
&&\qquad\leq C \big ( \big\| f \big\|_{L^1} +  \big\| R^L_jf \big\|_{L^1}
\big). \label{a53}
\end{eqnarray}
As pointed out in Remark \ref{rem3}, $(\ref{a53})$ implies that every
$\xi_{(k, \alpha)}f$ is in $h^1_k(\mathbb{H}^n)$ and allows the atomic
decomposition
\begin{eqnarray*}
\xi_{(k, \alpha)}f = \sum_{i} \lambda^{(k, \alpha)}_{i}\, a^{(k,
\alpha)}_{i}
\end{eqnarray*}
with
\begin{eqnarray*}
\sum_{i} \big| \lambda^{(k, \alpha)}_{i} \big| \leq C \bigg( \big\|
\xi_{(k, \alpha)}f \big\|_{L^1} + \sum_{j=1}^{2n} \big\|
\widetilde{R}^{[k]}_j(\xi_{(k, \alpha)}f) \big\|_{L^1} \bigg).
\end{eqnarray*}
Therefore we get
\begin{eqnarray*}
f = \sum_{(k, \alpha)} \sum_{i} \lambda^{(k, \alpha)}_{i}\, a^{(k,
\alpha)}_{i}
\end{eqnarray*}
and
\begin{eqnarray*}
\sum_{(k, \alpha)} \sum_{i} \big| \lambda^{(k, \alpha)}_{i} \big|
\leq C \bigg ( \big\| f \big\|_{L^1} + \sum_{j=1}^{2n} \big\| R^L_jf \big\|_{L^1} \bigg).
\end{eqnarray*}

For the reverse inequality, by Theorem \ref{thm2} and Lemma
\ref{lem18}, we have to check
\begin{eqnarray}\label{a54}
\big\| R^L_j a \big\|_{L^1} \leq C
\end{eqnarray}
for every $H^{1,q}_L$-atom $a$. Let $a$ be an $H^{1,q}_L$-atom supported
on a ball $B(g_0, r)$. Of course, we may assume that $q \leq 2$.
Since $R^L_j$ is bounded on $L^q(\mathbb{H}^n)$,
\begin{eqnarray*}
\big\| R^L_j a \big\|_{L^1(B(g_0, 2r))} \leq \big| B(g_0, 2r) \big|^{\frac{1}{q'}} \big\|
R^L_j a \big\|_{L^q} \leq C\, \big| B(g_0, r) \big|^{\frac{1}{q'}} \big\| a \big\|_{L^q}
\leq C.
\end{eqnarray*}
By the same argument as $(\ref{a45})$, we have
\begin{eqnarray*}
\big\| R^L_j a \big\|_{L^1(B(g_0, 2r)^c)} \leq C\, \big\| a \big\|_{L^1} \leq C.
\end{eqnarray*}
This proves $(\ref{a54})$, and the proof of Theorem \ref{thm3} is finished. \qed

\begin{remark}\label{rem6}
{\rm It follows from Theorem \ref{thm2}, Lemma \ref{lem18}, and $(\ref{a54})$
that the Riesz transforms $R^L_j$ are bounded from $H^1_L(\mathbb{H}^n)$
to $L^1(\mathbb{H}^n)$, and hence bounded from $H^1(\mathbb{H}^n)$
to $L^1(\mathbb{H}^n)$.}
\end{remark}

Finally we construct a counterexample which shows that the range
of $p$ ensuring the boundedness of the Riesz transforms $R^L_j$
can not be improved. The counterexample is similar to the one on
$\mathbb{R}^n$ given by Shen \cite{Shen}. The main difference is
the appearance of the function $\psi(g)$. Define the function
$\psi(g)$ by
\begin{eqnarray*}
\psi(g)= \frac{|x|^2}{|(x, t)|^2}, \qquad 0 \ne g=(x, t) \in \mathbb{H}^n.
\end{eqnarray*}
This function was introduced in \cite{Garofalo}. We simply remark
that $\psi$ is homogeneous of degree zero and $0\leq \psi\leq 1$.

\begin{lemma}\label{lem22} Let $1<q<\infty$. If $-\frac{Q}{q} <
\beta < \infty$, then $|g|^{\beta} \psi(g)$ belongs to $B_q$ class.
\end{lemma}

\begin{proof} Given a ball $B(h,r)$ with $|h| \leq 2r$, we have
\begin{eqnarray*}
\bigg ( \frac{1}{|B(h,r)|}\int_{B(h,r)} \big( |g|^{\beta}
\psi(g) \big)^q\, dg \bigg )^{\frac{1}{q}} \leq  C \bigg (
r^{-Q}\int_{B(0,4r)}|g|^{\beta q}\, dg \bigg)^{\frac{1}{q}}
\leq C r^{\beta}.
\end{eqnarray*}
On the other hand, there exists $h_0=(x_0, t_0) \in B(h,
\frac{r}{2})$ such that $|x_0|^2 \geq \frac{r^2}{16}$.
If $g \in B(h_0, \frac{r}{8})$, then $|g| \sim r$ and $\psi(g)
\sim 1$. Hence
\begin{eqnarray*}
\frac{1}{|B(h,r)|}\int_{B(h,r)} |g|^{\beta}\psi(g)\, dg \geq
\frac{1}{C}\,r^{-Q}\int_{B(h_0, \frac{r}{8})} |g|^{\beta}\psi(g)\, dg  \geq \frac{1}{C}\,
r^{\beta}.
\end{eqnarray*}
It follows that
\begin{eqnarray*}
\bigg( \frac{1}{|B(h,r)|} \int_{B(h,r)} \big  |g|^{\beta}
\psi(g) \big)^q\, dg \bigg)^{\frac{1}{q}}
\leq \frac{C}{|B(h,r)|}\int_{B(h,r)} |g|^{\beta} \psi(g)\, dg.
\end{eqnarray*}
In the case of $|h| > 2r$, we have $|g| \sim |h|$ when $g \in
B(h,r)$. Note that $P(g)= |g|^{2}\psi(g)$ is a nonnegative polynomial
of homogeneous degree two and satisfies
\begin{eqnarray*}
\max_{g\in B} P(g)\leq \frac{C}{|B|}\int_B P(g)\, dg\qquad
\text{for every ball}\ B\subset \mathbb H^n.
\end{eqnarray*}
Therefore,
\begin{eqnarray*}
\bigg( \frac{1}{|B(h,r)|}\int_{B(h,r)} \big( |g|^{\beta}
\psi(g) \big )^q\, dg \bigg)^{\frac{1}{q}}
&\leq& C\, |h|^{\beta -2} \max_{g \in B(h,r)} P(g) \\
&\leq& C\, |h|^{\beta -2}\bigg( \frac{1}{|B(h,r)|}\int_{B(h,r)}
P(g)\, dg \bigg)\\
&\leq& C \bigg( \frac{1}{|B(h,r)|}\int_{B(h,r)}
|g|^{\beta}\psi(g)\, dg \bigg)
\end{eqnarray*}
and Lemma \ref{lem22} is proved.
\end{proof}

Now we consider the nonnegative potential
\begin{eqnarray*}V(g)= |g|^{\beta -2}\psi(g),\qquad 0< \beta <2.
\end{eqnarray*}
We look for a radial solution $u(g)=f(|g|)$ of the equation
\begin{eqnarray}\label{a55}
-\Delta_{\mathbb H^n}u+ |g|^{\beta -2}\psi u=0.
\end{eqnarray}
By the following facts (cf. \cite{Garofalo}):
\begin{equation}\label{a56}
\Delta_{\mathbb H^n}(|g|) = \frac{Q-1}{|g|} \psi(g) \quad
\mathrm{and} \quad \big |\, \nabla_{\mathbb H^n}(|g|)\, \big |^2 =
\psi(g) \qquad\text{for}\ \ |g| \neq 0,
\end{equation}
it turns out that $f$ must satisfy
\begin{eqnarray*} f''(|g|)+ \frac{Q-1}{|g|}f'(|g|)- |g|^{\beta-2}f(|g|)=0
\qquad\text{for}\ \ |g| \neq 0.
\end{eqnarray*}
It is easy to verify that
\begin{eqnarray*}
v(g)= \sum^\infty_{m=0} \frac{|g|^{\beta m}}{m! \beta^{2m}
\Gamma(\frac{Q-2}{\beta}+m+1)}
\end{eqnarray*}
is a radial solution of the equation (\ref{a55}).

Let $\phi\in C^\infty_c(\mathbb H^n)$ such that $\phi=1$ for
$|g|\leq 1$. Set $u=\phi v$. Then
\begin{eqnarray*}
-\Delta_{\mathbb H^n}u + |g|^{\beta-2}\psi u = \eta,
\end{eqnarray*}
where $\eta =-2\nabla_{\mathbb H^n}v \cdot \nabla_{\mathbb H^n}
\psi -v \Delta_{\mathbb H^n}\phi \in C^\infty_c(\mathbb H^n)$.

Given $\frac{Q}{2}<q_{_1}< Q$, let $\beta = 2-\frac{Q}{q_{_1}}$.
By Lemma \ref{lem22}, $V(g)= |g|^{-\frac{Q}{q_{_1}}} \psi(g) \in
B_q$ for any $q < q_{_1}$. If $R^L_j$ is bounded on
$L^p(\mathbb{H}^n)$ for $p=p_{_1}$ where
$\frac{1}{p_{_1}}=\frac{1}{q_{_1}}-\frac{1}{Q}$, then
\begin{eqnarray*}
\big\| \nabla_{\mathbb{H}^n}u \big\|_{L^{p_{_1}}}
= \big\| \nabla_{\mathbb H^n}(- \Delta_{\mathbb{H}^n}
+V)^{-1} \eta \big\|_{L^{p_{_1}}} \leq C\, \big\| (-\Delta_{\mathbb
H^n}+V)^{-\frac{1}{2}} \eta \big\|_{L^{p_{_1}}}.
\end{eqnarray*}
It follows from $(\ref{a5})$ and $(\ref{a9})$ that
\begin{eqnarray*}
\big| (-\Delta_{\mathbb H^n}+V)^{-\frac{1}{2}} \eta(g) \big| \leq C
\int_{\mathbb{H}^n} \frac{|\eta (h)|}{|g^{-1}h|^{Q-1}}\, dh \leq
\frac{C}{(1+|g|)^{Q-1}}.
\end{eqnarray*}
Thus
\begin{eqnarray*}
\big\| \nabla_{\mathbb H^n}u \big\|_{L^{p_{_1}}}\, \leq C \big\|
(-\Delta_{\mathbb H^n}+V)^{-\frac{1}{2}} \eta \big\|_{L^{p_{_1}}}
< \infty.
\end{eqnarray*}
On the other hand, by $(\ref{a56})$ we have
\begin{eqnarray*}
|\nabla_{\mathbb{H}^n} u| \sim
|g|^{1-\frac{Q}{q_{_1}}}\psi^{\frac{1}{2}} =
|g|^{-\frac{Q}{p_{_1}}}\psi^{\frac{1}{2}},\qquad \text{as} \ \ g
\rightarrow 0.
\end{eqnarray*}
By the following formula about changing variables
\begin{eqnarray*}
\int_{\mathbb{H}^n} f(x,t)\, dx\, dt = \int_{S^{2n-1}}
\int_{-\frac{\pi}{2}}^{\frac{\pi}{2}} \int_0^{\infty} f \big(
r(\cos \theta)^{\frac{1}{2}}x', r^2 \sin \theta \big) (\cos
\theta)^{n-1} r^{Q-1}\, dr\, d\theta\, dx',
\end{eqnarray*}
where $S^{2n-1}$ is the unit sphere in $\mathbb{R}^{2n}$ (cf.
\cite{Coulhon}), it is easy to see that $\nabla_{\mathbb{H}^n}u
\not \in L^{p_{_1}}(\mathbb{H}^n)$. We have a contradiction.


\section{Results for stratified groups}

In this section, we state results for stratified groups.
We consistently use the same notations and terminologies as those 
in Folland and Stein's book \cite{Folland-Stein}.

Let $G$ be a stratified group of dimension $d$ with the Lie
algebra $\frak g$. This means that $\frak g$ is equipped
with a family of dilations $\{ \delta_r:\, r>0 \}$ and $\frak g$
is a direct sum $\bigoplus_{j=1}^m \frak g_j$ such that
$[\frak g_i, \frak g_j] \subset  \frak g_{i+j}$, $\frak g_1$
generates $\frak g$, and $\delta_r(X)= r^jX$ for $X \in
\frak g_j$. $Q=\sum_{j=1}^m j\,d_j$ is called the homogeneous
dimension of $G$, where $d_j= \dim \frak g_j$. $G$ is
topologically identified with $\frak g$ via the exponential
map $\exp: \frak g \mapsto G$ and $\delta_r$ is also
viewed as an automorphism of $G$. We fix a homogeneous norm of
$G$, which satisfies the generalized triangle inequalities
\begin{eqnarray*}
&& |xy| \leq \gamma (|x|+|y|)\qquad \text{for all} \ x,y \in G,\\
&& \big | |xy|-|x| \big | \leq \gamma |y|\qquad \text{for all} \ 
x,y \in G \ \text{with} \ |y| \leq \frac{|x|}{2},
\end{eqnarray*}
where $\gamma \geq 1$ is a constant. The ball of radius $r$
centered at $x$ is written by
\begin{eqnarray*}
B(x,r)=\{y \in G: \; |x^{-1}y |\, <r\}.
\end{eqnarray*}
The Haar measure on $G$ is simply the Lebesgue measure on $\mathbb R^d$ 
under the identification of $G$ with $\frak g$ and the identification 
of $\frak g$ with $\mathbb R^d$, where $d=\sum_{j=1}^m d_j$. 
The measure of $B(x,r)$ is 
\begin{eqnarray*}
\big| B(x,r) \big| = b\,r^{Q},
\end{eqnarray*}
where $b$ is a constant.

We identify $\frak g$ with $\frak g_L$, the Lie algebra
of left-invariant vector fields on $G$. Let $\{ X_j:\, j=1,
\cdots, d_1 \}$ be a basis of $\frak g_1$. The sub-Laplacian
$\Delta_G$ is defined by
\begin{eqnarray*}
\Delta_G= \sum^{d_1}_{j=1} X^2_j.
\end{eqnarray*}

The theory of Hardy spaces on homogeneous groups was studied by
Folland and Stein \cite{Folland-Stein}. Christ and Geller
\cite{Christ} gave the Riesz transforms characterization of the
Hardy space $H^1$ on stratified groups. The theory of local Hardy
spaces on stratified groups can be established as in Section 4.
In detail, we define the scaled local maximal functions
$\widetilde{M}_k f$'s by
\begin{align*}
&\widetilde{M}_{\phi,k}f(x)=\sup_{|x^{-1}y|<r\le 2^k} \big|
f*\phi_r(y) \big|,
&&\widetilde{M}^+_{\phi,k}f(x)=\sup_{0<r\le 2^k} \big| f*\phi_r(x) \big|,\\
&\widetilde{M}_{(N),k}f(x)=\sup_{\phi\in\mathscr S(G) \atop
\|\phi\|_{(N)}\le 1} \widetilde{M}_{\phi,k}f(x),
&&\widetilde{M}^+_{(N),k}f(x)=\sup_{\phi\in\mathscr S(G) \atop
\|\phi\|_{(N)}\le 1} \widetilde{M}^+_{\phi,k}f(x).
\end{align*}
The scaled local Hardy space $h^p_k(G)\, (p \leq 1)$ is defined by
\begin{eqnarray*}
h^p_k(G)= \big\{ f \in \mathscr {S}'(G):\; \widetilde{M}_{(N_p),\,
k}f \in L^p(G),\, N_p= [Q(\frac{1}{p} -1)] +1 \big\}
\end{eqnarray*}
with
\begin{eqnarray*}
\big\| f \big\|_{h^p_k} = \big\| \widetilde{M}_{(N_p),\, k}f
\big\|_{L^p}.
\end{eqnarray*}
Then we have
\begin{eqnarray*}
\big\| \widetilde{M}_{(N),\, k}f \big\|_{L^p} \sim \big\|
\widetilde{M}^+_{(N),\, k}f \big\|_{L^p} \sim \big\|
\widetilde{M}_{\phi,\, k}f \big\|_{L^p} \sim \big\|
\widetilde{M}^+_{\phi,\, k}f \big\|_{L^p},
\end{eqnarray*}
where $\phi$ is a commutative approximate identity and $N \geq
N_p$ is fixed. We have the atomic decomposition of $h^p_k(G)$ as
follows. Let $0<p \leq 1<q \leq \infty$. A function $a\in
L^q(G)$ is called an $h^{p,q}_k$-atom if the following
conditions hold:
\begin{itemize}
\item [\rm(i)] $\mathrm{supp}\, a \subset B(x_0,r)$, 
\item [\rm(ii)] $\big\| a \big\|_{L^{q}} \leq \big| B(x_0,r)
      \big|^{\frac{1}{q}-\frac{1}{p}}$, 
\item [\rm(iii)] if $\ r < 2^{k}$, then $\displaystyle\int_{B(x_0,r)} a(x) x^I\, dx =0\ $ for
      $\ d(I) < N_p$,
\end{itemize}
where $d(I)$ is the homogeneous degree of the monomial $x^I$. Then
$f \in h^p_k(G)$ if and only if $\,f$ can be written as $f= \sum_j
\lambda_j\, a_j$ converging in the sense of distributions and in
$h^p_k(G)$ norm, where $a_j$ are $h^{p,q}_k$-atoms and $\sum_j
|\lambda_j|^p < \infty$. Moreover,
\begin{eqnarray*}
\big\| f \big\|^p_{h^p_k} \sim \inf \bigg \{ \sum_j |\lambda_j|^p
\bigg \},
\end{eqnarray*}
where the infimum is taken over all atomic decompositions of $\,f$
into $h^{p,q}_k$-atoms. For $p=1$, $h^1_k(G)$ is also
characterized by the local Riesz transforms. Let
\begin{eqnarray*}
R_j= X_j (-\Delta_{G})^{-\frac{1}{2}},\qquad j=1, \cdots, d_1,
\end{eqnarray*}
be the Riesz transforms with the convolution kernel $R_j(x)$. A
function $f \in h^1_k(G)$ if and only if $f \in L^1(G)$ and
$\widetilde{R}^{[k]}_jf \in L^1(G),\, j=1, \cdots, d_1$, where the
local Riesz transforms are defined by $\widetilde{R}^{[k]}_jf= f
\ast R^{[k]}_j$ and $R^{[k]}_j(x)= \zeta (2^{-k} x)\, R_j(x)$ with
$\zeta \in C^{\infty}(G)$ satisfying $0 \leq \zeta (x) \leq 1$,
$\zeta (x)=1$ for $|x| < \frac{1}{2}$, and $\zeta (x)=0$ for $|x|> 1$. 
Moreover,
\begin{eqnarray*}
\big\| f \big\|_{h^1_k} \sim \big\| f \big\|_{L^1} +
\sum_{j=1}^{d_1} \big\| \widetilde{R}^{[k]}_jf \big\|_{L^1}.
\end{eqnarray*}

Let us consider the Schr\"odinger operator $L= -\Delta_{G}+V$,
where the potential $V$ is nonnegative and belongs to the reverse
H\"older class $B_{\frac{Q}{2}}$. We define the Hardy space
$H^1_L(G)$ associated with the Schr\"odinger operator $L$ by the
maximal function with respect to the semigroup $\big\{ T^L_s:\;
s>0 \big\} = \big\{ e^{-s L}:\; s>0 \big\}$. A function $f\in
L^1(G)$ is said to be in $H^1_L(G)$ if the maximal function $M^Lf$
belongs to $L^1(G)$, where $M^Lf(x)= \sup _{s>0} \big| T^L_sf(x)
\big|$. The norm of such a function is defined by $\big\| f
\big\|_{H^1_L} = \big\|M^Lf \big\|_{L^1}$. The atomic
decomposition of $H^1_L(G)$ is as follows. Let $1<q \leq \infty$.
A function $a\in L^q(G)$ is called an $H^{1,q}_L$-atom if the
following conditions hold:
\begin{itemize}
\item [\rm(i)] $\mathrm{supp}\, a \subset B(x_0,r)$, 
\item [\rm(ii)] $\big\| a \big\|_{L^{q}} \leq \big| B(x_0,r)
      \big|^{\frac 1q -1}$, 
\item [\rm(iii)] if $\ r < \rho(x_0),$ then
      $\displaystyle\int_{B(x_0,r)} a(x)\, dx =0$,
\end{itemize}
where the auxiliary function $\rho(x)= \rho(x, V)$ is defined as
before; that is,
\begin{eqnarray*}
\rho(x)= \sup_{r>0}\, \bigg \{ r:\; \frac{1}{r^{Q-2}}\int_{B(x,
r)}V(y)\, dy \leq 1 \bigg \}, \qquad x\in G.
\end{eqnarray*}
Let $f \in L^1(G)$ and $1<q \leq \infty$. Then $f \in H^1_L(G)$ if
and only if $\,f$ can be written as $f= \sum_j \lambda_j\, a_j$,
where $a_j$ are $H^{1,q}_L$-atoms, $\sum_j |\lambda_j|< \infty$,
and the sum converges in $H^1_L(G)$ norm. Moreover,
\begin{eqnarray*}
\big\| f \big\|_{H^1_L} \sim \inf \bigg\{ \sum_j |\lambda_j|
\bigg\},
\end{eqnarray*}
where the infimum is taken over all atomic decompositions of $\,f$
into $H^{1,q}_L$-atoms. The Hardy space $H^1_L(G)$ is also
characterized by the Riesz transforms $R^L_j$ associated with 
the Schr\"odinger operator $L$. These Riesz transforms are defined by
\begin{eqnarray*}
R^L_j= X_j L^{-\frac{1}{2}},\qquad j=1, \cdots, d_1.
\end{eqnarray*}
Each $R^L_j$ is bounded on $L^p(G)$ for $1<p\leq Q$ and bounded from 
$L^1(G)$ to $L^{1,\infty}(G)$. A function $f \in H^1_L(G)$ if and 
only if $\,f \in L^1(G)$ and $R^L_jf \in L^1(G),\, j=1, \cdots, d_1$.
Moreover,
\begin{eqnarray*}
\big\| f \big\|_{H^1_L} \sim \big\| f \big\|_{L^1} +
\sum_{j=1}^{d_1} \big\| R^L_jf \big\|_{L^1}.
\end{eqnarray*}
These results for the Hardy space $H^1_L$ on stratified groups 
can be proved by the same argument as for the Heisenberg group. 
In fact, the estimates in Section 3 keep true for stratified  groups.


\vskip 1cm

\flushleft Chin-Cheng Lin\\
          Department of Mathematics\\
          National Central University\\
          Chung-Li 320, Taiwan\\
          E-mail: clin@math.ncu.edu.tw
\vskip 0.75cm

\flushleft Heping Liu\\
          LMAM, School of Mathematical Sciences\\
          Peking University\\
          Beijing 100871, China\\
          E-mail: hpliu@pku.edu.cn
\vskip 0.75cm

\flushleft Yu Liu\\
          Department of Mathematics and Mechanics\\
          University of Science and Technology\\
          Beijing 100083, China\\
          E-mail: liuyu75@pku.org.cn

\end{document}